\documentclass{article}

\usepackage{amsmath,amsthm,amssymb,enumerate,epsfig,color,graphicx,epstopdf,tikz,rotating}

\usetikzlibrary{positioning,shapes}

\usepackage[all]{xy}

\usepackage{array}% http://ctan.org/pkg/array
\newcolumntype{C}[1]{>{\centering\arraybackslash$}p{#1}<{$}}

\usepackage{ifdraft}
\usepackage[colorlinks=true,linkcolor=blue,draft=false]{hyperref}
\usepackage{aliascnt}

\def\<{\langle}
\def\>{\rangle}

\theoremstyle{plain}
\newtheorem{theorem}{Theorem}[section]

\newaliascnt{lemma}{theorem}
\newtheorem{lemma}[lemma]{Lemma}
\aliascntresetthe{lemma}

\newaliascnt{proposition}{theorem}
\newtheorem{proposition}[proposition]{Proposition}
\aliascntresetthe{proposition}

\newaliascnt{corollary}{theorem}
\newtheorem{corollary}[corollary]{Corollary}
\aliascntresetthe{corollary}

\newaliascnt{conjecture}{theorem}

\aliascntresetthe{conjecture}

\theoremstyle{remark}

\newaliascnt{claim}{theorem}

\aliascntresetthe{claim}

\newtheorem*{claim*}{Claim}

\theoremstyle{definition}

\newaliascnt{definition}{theorem}
\newtheorem{definition}[definition]{Definition}
\aliascntresetthe{definition}

\newaliascnt{example}{theorem}

\aliascntresetthe{example}

\newaliascnt{notation}{theorem}

\aliascntresetthe{notation}

\title{On the growth of Artin--Tits monoids and the partial theta function}
\author{Ram\'on Flores, Juan Gonz\'alez-Meneses\footnote{Both authors partially supported by Spanish Project MTM2016-76453-C2-1-P and FEDER.}}
\date{August, 2018}

\begin{document}

\maketitle

%|<------------------------------------------------------------------------>|

\begin{abstract}
We present a new procedure to determine the growth function of a homogeneous Garside monoid, with respect to the finite generating set formed by the atoms. In particular, we present a formula for the growth function of each Artin--Tits monoid of spherical type (hence of each braid monoid) with res\-pect to the standard generators, as the inverse of the determinant of a very simple matrix.

Using this approach, we show that the exponential growth rates of the Artin--Tits monoids of type $A_n$ (positive braid monoids) tend to $3.233636\ldots$ as $n$ tends to infinity. This number is well-known, as it is the growth rate of the coefficients of the only solution $x_0(y)=-(1+y+2y^2+4y^3+9y^4+\cdots)$ to the classical partial theta function.
% $\displaystyle \sum_{k=0}^{\infty}{y^{k \choose 2}x^k}$.

We also describe the sequence $1,1,2,4,9,\ldots$ formed by the coefficients of $-x_0(y)$, by showing that its $k$th term (the coefficient of $y^k$) is equal to the number of braids of length $k$, in the positive braid monoid $A_{\infty}$ on an infinite number of strands, whose maximal lexicographic representative starts with the first generator $a_1$. This is an unexpected connection between the partial theta function and the theory of braids.
\end{abstract}

\section{Introduction}

A \emph{Garside monoid} is a cancellative monoid where greatest common divisors and least common multiples exist and some finiteness conditions are satisfied. The initial properties of Garside monoids were discovered by F. Garside in his famous paper \cite{Gar69}, which was concentrated in the particular case of braids. Garside theory was definitely established in the work of Paris and Dehornoy~\cite{Dehornoy-Paris, Dehornoy}, where in particular the notion of Garside group was defined as the group of fractions of a Garside monoid. Since then, Garside theory has become a remarkable topic in Combinatorial Group Theory, as very important families of groups and monoids are Garside (Artin--Tits groups of spherical type, torus link groups, groups related to solutions of quantum Yang-Baxter equations), and some well-known conjectures have been established with the help of Garside theory, like the $K(\pi,1)$ conjecture for finite complex reflection arrangements~\cite{Bessis}. We recall the basics we need of Garside theory at the beginning of \autoref{S:Garside}.

%Garside monoids~\cite{Dehornoy-Paris,Dehornoy} are a well-known class of monoids including Artin--Tits monoids of spherical type, and hence braid monoids. A %Garside monoid $M$ is cancellative, and the {\bf prefix order} (\mbox{$a\preccurlyeq b \Leftrightarrow \exists c\in M, \ ac=b$}) is a lattice order. This %means that given $a,b\in M$ there exists a unique greatest common divisor ($a\wedge b$) and a unique least common multiple ($a\vee b$), with respect to %$\preccurlyeq$.

The main objects of study of this paper are the \emph{growth functions} in Garside monoids, particularly in Artin--Tits monoids of spherical type, with respect to the generating set formed by the atoms. Our initial goal was to give a new procedure to compute the growth function of a homogeneous Garside monoid of finite type, which should yield a simple formula in well-known particular cases.

It was already shown by Deligne~\cite{Deligne} that the growth function $g_M(t)$ of a monoid $M$ is the inverse of a polynomial (called the {\it M\"obius polynomial} of $M$) if $M$ is an Artin--Tits monoid of spherical type, and his proof can be generalized to other Garside monoids (\autoref{C:Deligne}). Building on this, we develope in \autoref{S:Counting} a new and quite straightforward way of counting the number of elements of given length in a Garside monoid of finite type, or in other words, the coefficients of the growth function of the monoid, relative to the atoms.

A major consequence of this new way of counting is the following statement, which includes precise descriptions of the growth functions for all the Artin--Tits monoids of spherical type; these descriptions appear in the text of the paper as \autoref{T:typeA}, \autoref{T:typeB} and \autoref{T:typeD}, and we present them here in a unified way. Observe that, throughout the paper, we will assume that the number ${k\choose 2}=\frac{k(k-1)}{2}$ makes sense for every integer $k\geq 0$.

\textbf{Theorem.} Let $M$ an Artin--Tits monoid of spherical type. Then there exists a square matrix $\mathcal M$ such that $$g_M(t)=|\mathcal M|^{-1}.$$ Moreover:

\begin{itemize}

\item If $M$ is of type $A_n$, $\mathcal M$ is denoted $\mathcal M^{\mathcal A}_n$, has order $n+1$, and its entry $(i,j)$ equals $t^{j-i+1\choose 2}$ whenever $j-i+1\geq 0$, and 0 otherwise.

\item If $M$ is of type $B_n$, $\mathcal M$ is denoted $\mathcal M^{\mathcal B}_n$, has order $n+1$, its entry $(i,n+1)$ equals $t^{(n-i+1)^2}$ for all $i$, and its entry $(i,j)$ for $j\leq n$ equals $t^{j-i+1\choose 2}$ whenever $j-i+1\geq 0$, and 0 otherwise.

\item If $M$ is of type $D_n$, $\mathcal M$ is denoted $\mathcal M^{\mathcal D}_n$, has order $n$, its entry $(i,n)$ equals $2t^{n-i+1\choose 2}-t^{(n-i+1)(n-i)}$ for all $i$, and its entry $(i,j)$ for $j<n$ equals $t^{j-i+1\choose 2}$ whenever $j-i+1\geq 0$, and 0 otherwise.

\end{itemize}

The following are concrete examples of the above result, for $n=4$ or $n=5$:
$$
g_{A_4}(t)=\left|\begin{array}{ccccc} 1 & t & t^3 & t^6 & t^{10} \\ 1 & 1 & t & t^3 & t^6 \\ 0 & 1 & 1 & t & t^3 \\ 0 & 0 & 1 & 1 & t \\ 0 & 0 & 0 & 1 & 1\end{array}\right|^{-1} \qquad
g_{B_4}(t)=\left|\begin{array}{ccccc} 1 & t & t^3 & t^6 & t^{16} \\ 1 & 1 & t & t^3 & t^9 \\ 0 & 1 & 1 & t & t^4 \\ 0 & 0 & 1 & 1 & t \\ 0 & 0 & 0 & 1 & 1 \end{array}\right|^{-1}
$$
$$
g_{D_5}(t)=\left|\begin{array}{ccccc} 1 & t & t^3 & t^6 & 2t^{10}-t^{20} \\ 1 & 1 & t & t^3 & 2t^6-t^{12} \\ 0 & 1 & 1 & t & 2t^3-t^6 \\ 0 & 0 & 1 & 1 & 2t-t^2  \\ 0 & 0 & 0 & 1 & 1 \end{array}\right|^{-1}
$$

In~\cite{Bronfman}, Bronfman gave a recursive formula to compute (in an efficient way) the M\"obius polynomial of $M$, when $M=A_n$. Using the previous information, we are able to express the same M\"obius polynomial as the determinant of a matrix, recovering in particular (\autoref{T:Bronfman}) Bronfman's recursive formula as the expansion of this determinant along the first row.

Once we obtained these explicit formulae for the growth functions of the Artin--Tits monoids, we drove our attention to a more ambitious goal. Recall that given a monoid $M$, if we denote by $M_k$ the set of elements in $M$ of length $k$, the exponential growth rate of $M$ is defined as follows
$$
   \rho_M=\lim_{k\to \infty}{\sqrt[k]{|M_k|}}.
$$
In the case of an Artin--Tits monoid $M$ of spherical type, since $g_M(t)$ is the inverse of a polynomial, the number $\rho_M$ is inverse of the smallest root of the M\"obius polynomial, which is unique, real and positive, as shown by Jug\'e~\cite{Juge}. It is also known that, in the case of Artin--Tits monoids of type $A_n$ (the braid monoids), the limit of their growth rates as $n$ tends to infinity exists, and lies between $2.5$ and $4$ (see~\cite[Theorem 8]{VNB} and~\cite[Proposition 7.98]{Juge}). However, the precise description of this limit remained elusive for years, so we intended to compute it taking account of the information previously obtained in the article.

The strategy to obtain the limit implies the construction of a new bridge between (combinatorial) Monoid Theory and Real Analysis. Surprisingly, the coefficients of the leading root of the partial theta function (see \autoref{SS:partial}), which is defined in a pure analytical framework, can be computed counting braids. Consider the braid monoid $A_{\infty}$ on an infinite number of strands, which is the direct limit of the braid monoids $A_1\subset A_2\subset A_3\subset \cdots $ with the natural inclusions. We order the standard generators in the natural way ($a_1<a_2<a_3<\cdots$), and we consider the words in these generators ordered lexicographically. Every element in $A_{\infty}$ has a unique representative which is maximal with respect to this order. We will show the following:

{\bf \autoref{T:Sokal_sequence}.} {\it Let $x_0(y)$ be the only solution to the classical partial theta function $\displaystyle \sum_{k=0}^{\infty}{y^{k \choose 2}x^k}$, and let $\xi_0(y)=-x_0(y)=1+y+2y^2+4y^3+9y^4+\cdots$ For every $k\geq 0$, the coefficient of $y^k$ in the series $\xi_0(y)$ is equal to the number of braids of length $k$, in the monoid $A_{\infty}$, whose maximal lexicographic representative starts with $a_1$.}

 Then we use the fact that the growth rate of the coefficients $1,1,2,4,9,\ldots$ of $\xi_0(y)$ is known~\cite{Sokal} and, building on some results from~\cite{FlGo18}, we show the following:

{\bf \autoref{T:growth_limit}.} {\it Let $\displaystyle \rho=\lim_{n\to \infty}{\rho_{A_n}}$. Then $\rho=3.23363\ldots $ is the growth rate of the coefficients of $\xi_0(y)$. That is, $\rho$ is equal to the KLV-constant $q_{\infty}$.
}

It is overwhelming that this constant, that has a prominent role in different branches of Analysis (see \autoref{SS:constant}), appears here as a purely monoid-theoretic invariant associated to the braid monoid $A_n$, a fact that opens interesting perspectives of research. Moreover, we should remark that the constant can be determined with arbitrary precision, using for example the results of \cite{KLV03}. The reader interested in this relation between growth of braid monoids and the partial theta function can read \autoref{S:Garside} and \autoref{S:Counting}, and jump directly to \autoref{S:Growth_rates}.

%We finish our exposition by asking whether the growth rates of the Artin--Tits monoids of type $B_n$ and $D_n$ also tend to the same limit, as $n$ tends to %infinity.

{\bf Acknowledgements:} The second author thanks Yohei Komori for pointing out that the matrices determining the M\"obius polynomials could be simplified by taking all signs positive.

\section{Growth functions of homogeneous Garside monoids of finite type}

\label{S:Garside}

In this section we recall some basic facts of Garside theory and we discuss the growth of homogeneous Garside monoids of finite type.

In a Garside monoid $M$, which is cancellative, its elements form a lattice with respect to the {\it prefix order}, defined by $x\preccurlyeq y$ if $xz=y$ for some $z\in M$. The lattice property means that for every $a,b\in M$, there exist unique elements $a\wedge b$ and $a\vee b$, which are the greatest common divisor and the least common multiple, respectively, with respect to $\preccurlyeq$. It is important to notice that $\preccurlyeq$ is invariant under left multiplication and under left cancellation, that is, $x\preccurlyeq y$ if and only if $cx\preccurlyeq cy$ for every $c,x,y\in M$. This implies that $ca\wedge cb=c(a\wedge b)$ and $ca\vee cb=c(a\vee b)$ for every $a,b,c\in M$.

For every element in a Garside monoid $M$, the number of nontrivial factors that can be used to decompose it as a product of elements in $M$ is bounded above. The maximal number of factors for a given $a\in M$ is denoted $||a||$. It follows that every Garside monoid admits a special set of generators, called {\bf atoms}, which are those elements that cannot be decomposed as a product of two nontrivial elements: Every element $a\in M$ can be written as a product of $||a||$ atoms. A Garside monoid is said to be of {\bf finite type} if the set of atoms is finite.

Let us fix a Garside monoid $M$, and its set of atoms $\mathcal A$ as a set of generators. Given $a\in M$, its length $|a|$ is defined to be the length of the shortest word (in the atoms) representing $a$. If all relations in $M$ (written in terms of $\mathcal A$) are homogeneous, then $|a|=||a||$ is the word-length of any representative of $a$ as a word in the atoms. In this paper we shall only consider homogeneous Garside monoids of finite type, which include Artin--Tits monoids of spherical type.

For every integer $k$, let $M_k$ be the set of elements in $M$ of length $k$, and let $\alpha_k=|M_k|$ be its cardinal. Notice that $\alpha_0=1$, that $\alpha_k=0$ for $k<0$, and that $\alpha_k$ is a positive integer for $k\geq 0$, as we chose a finite set of generators. Let $t$ be an indeterminate. The {\bf growth function} (or growth series, or spherical growth series) of $M$ is defined to be:
$$
    g_M(t)=\sum_{k\geq 0}{\alpha_k t^k}
$$
It would be more precise to denote this function by $g_{M,\mathcal A}(t)$, as it depends on the generating set $\mathcal A$, but in this paper we shall only consider the set of atoms as generating set for $M$.

There are some well known results concerning the growth function of braid monoids, Artin--Tits monoids of spherical type, and more generally homogeneous Garside monoids of finite type. The first one is that $g_M(t)$ is the inverse of a polynomial (its M\"obius polynomial). This was shown by Deligne~\cite{Deligne} for Artin--Tits monoids of spherical type with their classical Garside structure. It was rediscovered by Bronfman~\cite{Bronfman}, generalizing it to a wider class of monoids, including homogeneous Garside monoids of finite type, and also rediscovered by Saito~\cite{Saito_2009}. We will show here a simple proof, in the spirit of Bronfman. Given a finite subset $S=\{s_1,\ldots,s_r\}\subset M$, let $\vee S =s_1\vee \cdots \vee s_r$. If $S=\emptyset$, we set $\vee S = 1\in M$.

\begin{theorem}[\cite{Deligne,Bronfman,Saito_2009}]\label{T:Deligne} Let $M$ be a homogeneous Garside monoid of finite type. Let $\mathcal A$ be its set of atoms, and let $\alpha_i=|M_i|$ be the number of elements of length $i$, written as a product of atoms. (Notice that $\alpha_i=0$ if $i<0$). Then, for every $k>0$, one has:
$$
 \sum_{S\subset \mathcal A } (-1)^{|S|} \alpha_{k-||\vee S||}=0
$$
\end{theorem}

%\begin{proof}
%This comes directly from the inclusion-exclusion principle. Let $\mathcal A=\{a_1,\ldots,a_m\}$. Given $S\subset A$, let $M^S_k=\{b\in M_k; \ s\preccurlyeq b, \; \forall s\in S\}$ the common multiples of the elements in $S$.  Then
%$$
%   M_k^\emptyset = M_k = M^{\{a_1\}}_k\cup \cdots \cup M^{\{a_r\}}_{k}.
%$$
%By the inclusion-exclusion principle, as $M^{\{a_{i_1}\}}_k\cap \cdots \cap M^{\{a_{i_r}\}}_k = M^{\{a_{i_1},\ldots,a_{i_r}\}}_k$, one has:
%$$
%|M_k^\emptyset| = \sum_{a_i\in \mathcal A} {\left|M^{\{a_i\}}_k\right|}-\sum_{\substack{S\subset \mathcal A \\ |S|=2}} {\left|M_k^{S}\right|}+\sum_{\substack{S\subset \mathcal A \\ |S|=3}} {\left|M_k^{S}\right|} +\cdots +
%(-1)^{m+1} \sum_{\substack{S\subset \mathcal A \\ |S|=m}} {\left|M_k^{S}\right|}.
%$$
%In other words,
%$$
%    \sum_{S\subset \mathcal A } {(-1)^{|S|} \left|M_k^{S}\right|}=0.
%$$
%Now notice that
%$$
%M_k^S=\{b\in M_k; \ (\vee S)\preccurlyeq b\}.
%$$
%As $M$ is cancellative and has homogeneous relations, the number of elements of length $k$ having $\vee S$ as a prefix, is precisely the number of elements of length $k-||\vee S||$. Therefore
%$|M_k^S|=\alpha_{k-||\vee S||}$.
%\end{proof}

\begin{proof}
Given $a\in M$, let $(aM)_k = \{b\in M;\ ||b||=k \mbox{ and } a\preccurlyeq b\}$, the set of elements of length $k$ which admit $a$ as a prefix. Let $\mathcal A=\{a_1,\ldots,a_m\}$. Every nontrivial element in $M$ must admit some atom as prefix, hence $M_k=
(a_1M)_k \cup \cdots \cup (a_mM)_k$.  By the inclusion-exclusion principle, it follows that
$$
\alpha_k = \sum_{\emptyset \neq S\subset \mathcal A} {(-1)^{|S|-1}\left|\bigcap_{a_i\in S}(a_iM)_k\right|}.
$$

As $M$ is cancellative and has homogeneous relations, one has
$$
(aM)_k = a\:M_{k-||a||}= \{ac\in M;\ ||c||=k-||a||\},
$$
for every $a\in M$. It follows that counting the elements in $(aM)_k$ is the same as counting the number of elements in $M$ of length $k-||a||$. That is, $|(aM)_k| = \alpha_{k-||a||}$.

Notice that given $S=\{x_1,\ldots,x_r\}\subset \mathcal A$, the common multiples of $x_1,\ldots,x_r$ are precisely the multiples of $x_1\vee \cdots \vee x_r$. Hence
$$
 (x_1M)_k\cap \cdots \cap (x_rM)_k = ((x_1\vee \cdots \vee x_r)M)_k = ((\vee S)M)_k,
$$
hence $\left|(x_1M)_k\cap \cdots \cap (x_rM)_k \right|=\alpha_{k-||\vee S||}$. Replacing this in the above expresion for $\alpha_k$, one gets:
$$
\alpha_k = \sum_{\emptyset \neq S\subset \mathcal A} {(-1)^{|S|-1}\alpha_{k-||\vee S||}},
$$
which is precisely what we wanted to show, as $\alpha_k=\alpha_{k-0}=\alpha_{k-||\vee \emptyset||}$.
\end{proof}

\begin{corollary}[\cite{Deligne,Bronfman,Saito_2009}]\label{C:Deligne}
If $M$ is a homogeneous Garside monoid of finite type, the growth function of $M$ is the inverse of a polynomial. Namely, if $\mathcal A\subset M$ is the set of atoms,
$$
    g_M(t)=\left( \sum_{S\subset \mathcal A } (-1)^{|S|}\: t^{||\vee S||} \right)^{-1}
$$
\end{corollary}

\begin{proof}
Let $N=||\vee \mathcal A||$. Notice that, for every $k>0$, the formula in \autoref{T:Deligne} is a sum of terms of the form $\pm \alpha_{k-i}$ for some $i\leq N$. Collecting the terms with the same value of $i$, one obtains:
\begin{equation}\label{E:coeff_recurrence_relation}
0=\sum_{S\subset \mathcal A } (-1)^{|S|} \alpha_{k-||\vee S||}= \sum_{i=0}^{N}c_i \alpha_{k-i}.
\end{equation}
It is important to notice that the coefficient $c_i$ does not depend on $k$: it only depends on the number and size of subsets $S\subset \mathcal A$ such that $||\vee S||=i$. Since $c_0=1$, we have a recurrence relation which determines the sequence $\{\alpha_k\}_{k\geq 0}$:
\begin{equation}\label{E:coeff_recurrence_relation_explicit}
\alpha_k=-\sum_{i=1}^N{c_i\alpha_{k-i}},
\end{equation}
Each $\alpha_k$ is thus a linear combination of $\alpha_{k-1},\ldots,\alpha_{k-N}$, where the coefficients $-c_1,\cdots,-c_N$ are fixed (they depend only on the monoid $M$).

It is easy to check that the coefficients of a power series (in this case $g_M(t)$) satisfy such a recurrence relation if and only if the power series is the inverse of a polynomial, namely:
$$
   g_M(t) = \left( \sum_{i=0}^N{c_i\: t^i} \right)^{-1}.
$$
Notice that the polynomial in the above formula is the one obtained from the right hand side of~(\ref{E:coeff_recurrence_relation}) when replacing $\alpha_{k-i}$ with $t^i$ for $i=0,\ldots,N$. Therefore, by the second equality in~(\ref{E:coeff_recurrence_relation}), we obtain:
$$
   \sum_{S\subset \mathcal A } (-1)^{|S|} t^{||\vee S||}= \sum_{i=0}^{N}c_i t^{i}.
$$
Hence
$$
    g_M(t) = \left(\sum_{S\subset \mathcal A} (-1)^{|S|} t^{||\vee S||}\right)^{-1}.
$$
\end{proof}

We can then write $g_M(t)=\frac{1}{H_M(t)}$, where $H_M(t)=\sum_{S\subset \mathcal A} (-1)^{|S|} t^{||\vee S||}$, the M\"obius polynomial of $M$, has degree $||\vee \mathcal A||$. In the case of the braid monoid on $n$ strands, this degree is $n\choose 2$.

We remark that the recurrence relation~(\ref{E:coeff_recurrence_relation_explicit}) gives a very fast way to obtain $\alpha_k$ for any given $k>0$, starting from $\alpha_0=1$ and $\alpha_i=0$ for $i<0$, provided that the coefficients $c_i$ are known. But in order to obtain this recurrence relation~(\ref{E:coeff_recurrence_relation_explicit}) explicitly, using the formula in \autoref{T:Deligne}, one needs to collect all terms involving each $\alpha_{k-i}$. This means to run through all subsets of $\mathcal A$. Therefore, computing the recurrence relation~(\ref{E:coeff_recurrence_relation_explicit}) using \autoref{T:Deligne} has exponential complexity with respect to the number of atoms.

Nevertheless, when the number of atoms is small, one can compute the growth function without much problem using the formula in \autoref{C:Deligne}. For instance, we include here the (already known) growth functions of all Artin--Tits groups of spherical type which are not of type $A$, $D$, $E$.

$
g_{E_6}(t)=\left(\begin{array}{c}1-6t+10t^2-10t^4+5t^5-4t^6+3t^7+ \\ +4t^{10}-2t^{11}+t^{12}-t^{15}-2t^{20}+t^{36}\end{array}\right)^{-1}
$

$
g_{E_7}(t)=\left(\begin{array}{c}1-7t+15t^2-5t^3-16t^4+12t^5-3t^6+8t^7-3t^8-3t^9+ \\
  +6t^{10}-5t^{11}+t^{12}-3t^{15}+t^{16}-2t^{20}+2t^{21}+t^{30}+t^{36}-t^{63}\end{array}\right)^{-1}
$

$
g_{E_8}(t)=\left(\begin{array}{c}1-8t+21t^2-14t^3-21t^4+28t^5-7t^6+12t^7-8t^8-10t^9+ \\
+10t^{10}-12t^{11} +7t^{12} +2t^{13} -t^{14} -3t^{15}
+2t^{16}-2t^{20}+ \\ +6t^{21} -t^{22}-t^{23} -t^{28} +t^{30} +t^{36} -t^{37}-t^{42} -t^{63} +t^{120} \end{array}\right)^{-1}
$

$
g_{F_4}(t)=\left( 1-4t+3t^2+2t^3-t^4-2t^9+t^{24} \right)^{-1}
$

$
g_{H_3}(t)=\left( 1-3t+t^2+t^3+t^5-t^{15} \right)^{-1}
$

$
g_{H_4}(t)=\left( 1-4t+3t^2+2t^3-t^4+t^5-2t^6-t^{15}+t^{60} \right)^{-1}
$

$
g_{I_2(p)}(t)=\left( 1-2t+t^p\right)^{-1}
$

The last case is valid for $p\geq 3$ (even for $p=2$, which does not yield an irreducible monoid), although it is usual to denote $I_2(3)=A_2$ and $I_2(4)=B_2$.

Concerning the remaining types, Bronfman~\cite{Bronfman} gave a procedure to compute the recurrence relation~(\ref{E:coeff_recurrence_relation_explicit}), or equivalently the M\"obius polynomial $H_M(t)$, in the case of braid monoids (Artin--Tits monoids of type $A$). Namely, he gave a recurrence relation expressing the polynomial $H_{A_n}(t)$ in terms of $H_{A_1}(t), H_{A_2}(t), \ldots,$ $H_{A_{n-1}}(t)$. Using Bronfman's formula, one has a polynomial time algorithm (with respect to $n$) to compute the growth function of $A_n$.

In \autoref{T:Bronfman} we will recover Bronfman's recurrence relation as an immediate consequence of our formula for $g_{A_n}(t)$. And we will also obtain new recurrence relations involving Artin--Tits monoids of types $B$ and $D$. This will be done in the next section.

\section{Counting elements in a homogeneous Garside monoid} \label{S:Counting}

We will now present a more straightforward way to count the number of elements of given length in the monoid $M$, that is, the coefficients of the growth function $g_M(t)$, without using the formula in \autoref{T:Deligne}.

Recall that $M$ is a homogeneous Garside monoid of finite type, and that $M_k$ is the set of elements of length $k$ (as words in the atoms). Let $\mathcal A=\{a_1,\ldots,a_n\}$ be the set of atoms. We will study the elements in $M_k$ by choosing a suitable word representing each element. Namely, we define the {\bf lex-representative} of an element of $M$, as its biggest representative in lexicographical order with $a_1<a_2<\cdots <a_n$. Notice that the number of lex-representatives of length $k$ is precisely $\alpha_k$.

It is clear that $M_0=\{1\}$. For $k>0$, we will be able to count the number of elements in $M_k$ using a stratification of these sets by suitable subsets.
\begin{definition}
For $k\geq 0$ and $i=1,\ldots,n$, we define $L_k^{(i)}$ to be the set of elements in $M_k$ whose lex-representative starts with $a_i$. That is,
$$
   L_k^{(i)}=\{b\in M_k;\quad a_i\preccurlyeq b, \quad a_{i+1},\ldots,a_{n}\not\preccurlyeq b \}.
$$
\end{definition}

It is then clear that, for $k>0$, \quad $M_k= L_k^{(1)}\sqcup L_k^{(2)}\sqcup \cdots \sqcup L_k^{(n)}$.

The number of elements in each $L_k^{(i)}$ is not easy to compute directly. But we can define some related subsets of $M_k$ which will allow us to perform the computation:

\begin{definition}
Given $k\geq 0$ and $1\leq i\leq j \leq n$, let
$$
 U_{k}^{(i,j)}=\{b\in M_k;\quad a_i,\ldots,a_{j}\preccurlyeq b, \quad a_{j+2},\ldots, a_n\not \preccurlyeq b \}
$$
\end{definition}

Notice that if $k>0$, every $b\in U_{k}^{(i,j)}$ admits all atoms $a_{i},\ldots, a_j$ as prefixes, which is equivalent to $a_i\vee \cdots \vee a_j\preccurlyeq b$. As $a_{j}\preccurlyeq b$, the lex-representative of $b$ starts either with $a_{j}$ or with an atom of bigger index. But $b$ does not admit $a_t$ as prefix for $t\geq j+2$. Hence, the lex-representative of $b$ starts with either $a_{j}$ or $a_{j+1}$.

Therefore, $b\in U_{k}^{(i,j)}$ if and only if it has length $k$, its lex-representative starts with either $a_j$ or $a_{j+1}$, and $a_{i}\vee \cdots \vee a_j \preccurlyeq b$.

\begin{lemma}\label{L:L in terms of U}
For every $k\geq 0$ and $i=1,\ldots,n$ one has:
$$
  \left|L_k^{(i)}\right|=\sum_{j=i}^{n}{(-1)^{j-i}\left|U_{k}^{(i,j)}\right|}
$$
\end{lemma}

\begin{proof}
The case $k=0$ is trivial, as all sets are empty. We will then assume that $k>0$.

For every element $b\in U_{k}^{(i,j)}$, its lex-representative starts with either $a_{j}$ or $a_{j+1}$. This depends on whether $a_{j+1}\preccurlyeq b$ or not. We can then split $U_{k}^{(i,j)}$ into two disjoint subsets, if $j<n$:
\begin{eqnarray*}
   U_{k}^{(i,j)} & = &\{b\in M_k;\quad a_i,\ldots,a_{j}\preccurlyeq b, \quad a_{j+1},\ldots, a_n\not\preccurlyeq b \} \\
   & \bigsqcup & \{b\in M_k;\quad a_i,\ldots,a_{j+1}\preccurlyeq b, \quad a_{j+2},\ldots, a_n\not\preccurlyeq b \}.
\end{eqnarray*}
If $j=n$ we just have:
$$
   U_{k}^{(i,n)}=\{b\in M_k;\ a_i,\ldots,a_n\preccurlyeq b\}.
$$
It is important to notice that if $j_1$ and $j_2$ are not consecutive, the sets $U_k^{(i,j_1)}$ and $U_k^{(i,j_2)}$ are disjoint. Hence, we can consider the disjoint union $U_{k}^{(i,i)}\sqcup U_{k}^{(i,i+2)}\sqcup U_{k}^{(i,i+4)} \sqcup \cdots$, and also the disjoint union $U_{k}^{(i,i+1)}\sqcup U_{k}^{(i,i+3)}\sqcup U_{k}^{(i,i+5)} \sqcup \cdots$.

Suppose first that $n-i$ is odd. In this case:
$$
  U_{k}^{(i,i)}\sqcup U_{k}^{(i,i+2)}\sqcup \cdots \sqcup U_{k}^{(i,n-1)}= \bigsqcup_{j=i}^{n}\{b\in M_k;\quad a_i,\ldots,a_j\preccurlyeq b, \quad a_{j+1},\ldots,a_n\not\preccurlyeq b\}.
$$
And also:
$$
  U_{k}^{(i,i+1)}\sqcup U_{k}^{(i,i+3)}\sqcup \cdots \sqcup U_{k}^{(i,n)}= \bigsqcup_{j=i+1}^{n}\{b\in M_k;\quad a_i,\ldots,a_j\preccurlyeq b, \quad a_{j+1},\ldots,a_n\not\preccurlyeq b\}.
$$
Therefore, the former disjoint union contains the latter, and we have:
$$
 \left(U_{k}^{(i,i)}\sqcup U_{k}^{(i,i+2)}\sqcup \cdots \sqcup U_{k}^{(i,n-1)}\right)\setminus \left(U_{k}^{(i,i+1)}\sqcup U_{k}^{(i,i+3)}\sqcup \cdots \sqcup U_{k}^{(i,n)}\right)
$$
$$
 =\{b\in M_k;\quad a_i\preccurlyeq b, \quad a_{i+1},\ldots,a_n\not\preccurlyeq b\}=L_k^{(i)},
$$
which implies the formula in the statement.

If $n-i$ is even, the argument is analogous. We have:
$$
  U_{k}^{(i,i)}\sqcup U_{k}^{(i,i+2)}\sqcup \cdots \sqcup U_{k}^{(i,n)}= \bigsqcup_{j=i}^{n}\{b\in M_k;\quad a_i,\ldots,a_j\preccurlyeq b, \quad a_{j+1},\ldots,a_n\not\preccurlyeq b\}.
$$
And also:
$$
  U_{k}^{(i,i+1)}\sqcup U_{k}^{(i,i+3)}\sqcup \cdots \sqcup U_{k}^{(i,n-1)}= \bigsqcup_{j=i+1}^{n}\{b\in M_k;\quad a_i,\ldots,a_j\preccurlyeq b, \quad a_{j+1},\ldots,a_n\not\preccurlyeq b\}.
$$
Hence, removing the latter union from the former yields $L_k^{(i)}$, and the formula in the statement also holds in this case.
\end{proof}

We are mainly interested in Artin--Tits monoids of type $A$, $B$ and $D$ (also called Artin--Tits monoids of type $A_n$, $B_n$ and $D_n$). In those cases, we will be able to describe the sizes of each $L_k^{(i)}$ and each $U_{k}^{(i,j)}$ in terms of the sizes of the following sets:

\begin{definition}
For $k\geq 0$ and $i=0,\ldots,n+1$, let
$$
   M_k^{(i)}=\{b\in M_k;\quad a_{i+1},\ldots,a_{n}\not\preccurlyeq b\}.
$$
\end{definition}

Notice that $M_0^{(i)}=\{1\}$ for every $i$. On the other hand, if $k>0$, the condition $a_{i+1},\ldots,a_{n}\not\preccurlyeq b$ just means that the lex-representative of $b$ starts with an atom from $\{a_1,\ldots,a_i\}$. Hence, if $k>0$:
$$
    M_k^{(i)}=L_k^{(1)}\sqcup L_k^{(2)}\sqcup \cdots \sqcup L_k^{(i)}.
$$
Notice that $M_k^{(0)}=\emptyset$ and that $M_k^{(n+1)}=M_k^{(n)}=M_k$.

It is clear, by definition, that $M_k^{(i-1)}$ is a subset of $M_k^{(i)}$, and that
$$
    L_k^{(i)}=M_k^{(i)}\setminus M_k^{(i-1)}.
$$
Hence:
$$
    \left|L_k^{(i)}\right|=\left|M_k^{(i)}\right|- \left|M_k^{(i-1)}\right|.
$$
In order to avoid cumbersome notation in the following formulae, we will denote $m_{k,i}=\left|M_k^{(i)}\right|$ and $u_{k,i,j}=\left|U_{k}^{(i,j)}\right|$. Then we have:

\begin{proposition}\label{L:M in terms of U}
For $k\geq 0$ and $i=1,\ldots,n+1$, one has:
$$
    m_{k,i}=m_{k,i-1}+\sum_{j=i}^n{(-1)^{j-i}u_{k,i,j}}.
$$
\end{proposition}

\begin{proof}
If $i\leq n$, this is just \autoref{L:L in terms of U}, as $m_{k,i}-m_{k,i-1}=\left|L_k^{(i)}\right|$.  If $i=n+1$, the formula reads $m_{k,n+1}=m_{k,n}$, which is true since $M_k^{(n+1)}=M_k^{(n)}$.
\end{proof}

If the monoid $M$ is an Artin--Tits monoid of type $A_n$, $B_n$ or $D_n$, we will be able to describe each number $u_{k,i,j}$ in terms of some $m_{l,t}$, with $l<k$. Replacing this in the formula of \autoref{L:M in terms of U}, we will obtain a recurrence relation for the numbers $m_{k,i}$, which is precisely what we need to compute the number $\alpha_k=m_{k,n}=m_{k,n+1}$.

We will then compute a table of the form

\begin{center}
\begin{tabular}{|c|c|c|c|c|}
\hline
 $m_{0,1}$ & $m_{0,2}$ & \ldots & $m_{0,n}$ & $m_{0,n+1}$ \\ \hline
 $m_{1,1}$ & $m_{1,2}$ & \ldots & $m_{1,n}$ & $m_{1,n+1}$ \\ \hline
 $m_{2,1}$ & $m_{2,2}$ & \ldots & $m_{2,n}$ & $m_{2,n+1}$ \\ \hline
 $m_{3,1}$ & $m_{3,2}$ & \ldots & $m_{3,n}$ & $m_{3,n+1}$ \\ \hline
  \vdots & \vdots &         & \vdots & \vdots \\ \hline
 $m_{k,1}$ & $m_{k,2}$ & \ldots & $m_{k,n}$ & $m_{k,n+1}$ \\ \hline
  \vdots & \vdots &         & \vdots & \vdots \\ \hline
\end{tabular}
\end{center}
The last two columns of this table will be identical, and they will contain the numbers $m_{k,n}=m_{k,n+1}=\alpha_k$, that is, the number of elements in $M_k$.

We will be able to compute each number in the above table, as a (signed) sum of at most $n$ of the previous elements. The method for computing each $m_{k,i}$ will produce the new formulae for the M\"obius polynomial of the monoid, and a better understanding of its growth rate, in the cases in which $M$ is an Artin--Tits monoid of type $A_n$, $B_n$ or $D_n$.

\section{Artin--Tits monoid of type A}

\subsection{Counting elements in the monoid (type A)}

Let $A_n$ be an Artin--Tits monoid of type $A$ with $n$ standard generators. That is, the positive braid monoid with $n+1$ strands. Its standard presentation is the following:
$$
A_n=\left\langle a_1,\ldots,a_n \; \left| \begin{array}{cl} a_ia_j=a_ja_i, & |i-j|>1 \\ a_ia_ja_i=a_ja_ia_j, & |i-j|=1 \end{array}   \right.\right\rangle
$$
It is well--known~\cite{Paris}  that if $1\leq i\leq j\leq n$, the submonoid of $A_n$ generated by $\{a_i,\ldots,a_j\}$ is again an Artin--Tits monoid of type $A_{j-i+1}$. Therefore, the least common multiple of the standard generators (in both the submonoid and the monoid) is~\cite[Lemma 3.1]{Paris2}:
$$
  a_i\vee \cdots \vee a_j = a_i (a_{i+1}a_i)(a_{i+2}a_{i+1}a_{i})\cdots (a_{j}a_{j-1}\cdots a_i).
$$
Hence:
$$
   ||a_{i}\vee \cdots \vee a_j||= {j-i+2 \choose 2}.
$$
Also, for $t \geq j+2$, the atom $a_t$ commutes with $a_{i},\ldots,a_j$, hence it commutes with the whole element $a_i\vee \cdots \vee a_j$. Therefore:
$$
   (a_{i}\vee \cdots \vee a_j)\vee a_t = (a_{i}\vee \cdots \vee a_j)\: a_t.
$$
Because of these two properties, we can explicitly determine the numbers $u_{k,i,j}$ as follows.

\begin{lemma}\label{L:U case A}
Let $M$ be the Artin--Tits monoid of type $A_n$, and let $1\leq i \leq  j \leq n$. One has
$$
   u_{k,i,j}=m_{k-{j-i+2\choose 2},j+1}
$$
%$$
%   \left|U_{k,i-j}^{(i)}\right|= \left|M_{k-{j+1\choose 2}}^{(i-j)}\right|
%$$
\end{lemma}

\begin{proof}
We have
\begin{eqnarray*}
 U_{k}^{(i,j)} & = & \{b\in M_k;\quad a_i,\ldots,a_{j}\preccurlyeq b, \quad a_{j+2},\ldots, a_n\not\preccurlyeq b \} \\
 & = & \{b\in M_k;\quad  a_i\vee\cdots \vee a_j\preccurlyeq b, \quad a_{j+2}, \ldots, a_n \not\preccurlyeq b \} \\
 & = & \{(a_{i}\vee \cdots \vee a_j)c \in M_k;\quad a_{j+2},\ldots,a_{n}\not\preccurlyeq (a_{i}\vee \cdots \vee a_j)c \}
\end{eqnarray*}
Notice that for every $t\geq j+2$, one has $a_t\preccurlyeq (a_{i}\vee \cdots \vee a_j)c$ if and only if $(a_{i}\vee \cdots \vee a_j)\vee a_t\preccurlyeq (a_{i}\vee \cdots \vee a_j)c$, that is  $(a_{i}\vee \cdots \vee a_j)\: a_t\preccurlyeq (a_{i}\vee \cdots \vee a_j)c$, which is equivalent to $a_t\preccurlyeq c$. Hence
$$
 U_{k}^{(i,j)} = \{(a_{i}\vee \cdots \vee a_j)c \in M_k;\ a_{j+2},\ldots,a_{n}\not\preccurlyeq c \}.
$$
As $M$ is cancellative and homogeneous, the set of elements $(a_{i}\vee \cdots \vee a_j)c$ having length $k$ is in bijection with set of elements $c$ having length $k-||a_{i}\vee \cdots \vee a_j||=k-{j-i+2\choose 2}$. Therefore
$$
\left|U_{k}^{(i,j)} \right|= \left|\{c \in M_{k-{j-i+2\choose 2}};\ a_{j+2},\ldots,a_{n}\not\preccurlyeq c \}\right| = \left| M_{k-{j-i+2\choose 2}}^{(j+1)} \right|
$$
\end{proof}

\begin{corollary}\label{C:M recurrence relation type A}
In the Artin--Tits monoid $A_n$, for $k\geq 0$ and $i=1,\ldots,n+1$, one has:
$$
   m_{k,i}=m_{k,i-1}+\sum_{j=i}^{n}{(-1)^{j-i}m_{k-{j-i+2\choose 2},j+1}}
$$
\end{corollary}

\begin{proof}
This is a direct consequence of \autoref{L:M in terms of U} and \autoref{L:U case A}.
\end{proof}

The above recurrence relation allows us to compute, in a very efficient way, a table whose entries are $m_{k,i}$ for $i=1,\ldots,n+1$ and $k\geq 0$, for the monoid $A_n$. We start the table with the first row:
$$
   m_{0,1}=m_{0,2}=\cdots = m_{0,n+1}=1
$$
The second row can then be computed, from left to right, using the recurrence relation of \autoref{C:M recurrence relation type A}. That is, we compute $m_{1,1}=m_{1,0}+m_{0,2}=0+1=1$, then $m_{1,2}=m_{1,1}+m_{0,3}=1+1=2$, then $m_{1,3}=2+1=3$ and so on, up to $m_{1,n}=(n-1)+1=n$ and $m_{1,n+1}=n$. In the same way, we can compute each new row from the previous ones, starting from the leftmost entry, using the recurrence relation of \autoref{C:M recurrence relation type A}. For instance, we can see in \autoref{F:table A_3} the first seven rows of the table corresponding to the monoid $A_3$. Notice for instance that $m_{6,1}=m_{6,0}+m_{5,2}-m_{3,3}+m_{0,4}=0+51-19+1=33$, or that $m_{6,2}= m_{6,1}+m_{5,3}-m_{3,4}= 33+94-19 = 108$.

\begin{figure}[ht]
\begin{center}
\begin{tabular}{|@{}c@{}||c|c|c|c|}
\hline $\begin{array}{ccc} & & \\ &  & i \\ & k &  \end{array}$ & 1 & 2 & 3 & 4 \\
\hline \hline 0 & 1 & 1 & 1 & 1 \\
\hline 1 & 1 & 2 & 3 & 3 \\
\hline 2 & 2 & 5 & 8 & 8 \\
\hline 3 & 4 & 11 & 19 & 19 \\
\hline 4 & 8 & 24 & 43 & 43 \\
\hline 5 & 16 & 51 & 94 & 94 \\
\hline 6 & 33 & 108 & 202 & 202 \\
\hline
\end{tabular}
\end{center}
\caption{A table containing $m_{k,i}$ for the Artin--Tits monoid $A_3$, for $k\leq 6$.}
\label{F:table A_3}
\end{figure}

Now recall that the rightmost column (and also the adjacent column, which is identical) contains precisely the coefficients of the growth function $g_{A_n}(t)$, as $m_{k,n+1} = m_{k,n} = \left|M_k^{(n)}\right| = \left|M_k\right| = \alpha_k$. For instance, the table in \autoref{F:table A_3} tells us that there are $202$ elements of length 6 in the monoid $A_3$. In other words, there are $202$ positive braids of length 6 with 4 strands.

\subsection{A new formula for the growth function (type A)}

From the recurrence relation given in \autoref{C:M recurrence relation type A} to compute the table of $m_{k,i}$'s, we will be able to provide a new formula for the growth function of the Artin--Tits monoid $A_n$. We will also do the same for types $B_n$ and $D_n$ in \autoref{S:B_and_D}.

Notice that every new entry of the table is obtained from the previous ones by a linear combination with coefficients $0$ or $\pm 1$. Moreover, in order to compute the $k$th row of the table, one just needs to use values from the previous ${n+1 \choose 2}$ rows.

More precisely, for $k\geq 1$ let $\mathbf v_{k-1}$ be a column vector whose entries correspond to the entries of rows $k-1$ to $k-{n+1\choose 2}$in the table (we consider $m_{t,i}=0$ if $t<0$). Namely
$$
  \mathbf v_{k-1} = \left(m_{k-1,1}\cdots m_{k-1,n+1} \: m_{k-2,1}\cdots m_{k-2,n+1} \cdots m_{k-{n+1\choose 2},1} \cdots m_{k-{n+1\choose 2},n+1}\right)^t
$$

For instance, if $M=A_3$, we can check from the table in \autoref{F:table A_3} that:
$$
  \mathbf v_0=\left(1 \ 1 \ 1 \ 1 \ | \ 0 \ 0 \ 0 \ 0 \ | \ 0 \ 0 \ 0 \ 0 \ | \ 0 \ 0 \ 0 \ 0 \ | \ 0 \ 0 \ 0 \ 0 \ | \ 0 \ 0 \ 0 \ 0  \right)^t
$$
$$
  \mathbf v_1=\left(1 \ 2 \ 3 \ 3 \ | \ 1 \ 1 \ 1 \ 1 \ | \ 0 \ 0 \ 0 \ 0 \ | \ 0 \ 0 \ 0 \ 0 \ | \ 0 \ 0 \ 0 \ 0 \ | \ 0 \ 0 \ 0 \ 0  \right)^t
$$
$$
  \mathbf v_2=\left(2 \ 5 \ 8 \ 8 \ | \ 1 \ 2 \ 3 \ 3 \ | \ 1 \ 1 \ 1 \ 1 \ | \ 0 \ 0 \ 0 \ 0 \ | \ 0 \ 0 \ 0 \ 0 \ | \ 0 \ 0 \ 0 \ 0 \right)^t
$$
And so on. Just to give an example with nonzero entries, we have:
$$
  \mathbf v_6=\left(33 \ 108 \ 202 \ 202 \ | \ 16 \ 51 \ 94 \ 94 \ | \ 8 \ 24 \ 43 \ 43 \ | \ 4 \ 11 \ 19 \ 19 \ | \ 2 \ 5 \ 8 \ 8 \ | \ 1 \ 2 \ 3 \ 3  \right)^t
$$

From the above arguments, every entry of $\mathbf v_k$ is a linear combination of the entries of $\mathbf v_{k-1}$, for every $k>0$. And the coefficients of the linear combination do not depend on $k$. Hence, we have:

\begin{lemma}\label{L:matrix A}
There is a square matrix $\mathcal A$ with ${n+1\choose 2}(n+1)$ rows, whose entries belong to $\{0,1,-1\}$ such that for every $k\geq 1$
$$
      \mathcal A \mathbf v_{k-1} = \mathbf v_k
$$
\end{lemma}

\begin{proof}
Just to give an idea of how the matrix $\mathcal A$ looks like, here is $\mathcal A$ when $n=3$:

$$
\mathcal A={\arraycolsep 3pt
\left( \begin{array}{c|c|c|c|c|c}
   {\arraycolsep 3pt \scriptsize \begin{array}{cccc} 0 & 1 & 0 & 0 \\ 0 & 1 & 1 & 0 \\ 0 & 1 & 1 & 1 \\ 0 & 1 & 1 & 1 \end{array}} & {\arraycolsep 3pt \scriptsize \begin{array}{cccc} 0 & 0 & 0 & 0 \\ 0 & 0 & 0 & 0 \\ 0 & 0 & 0 & 0 \\ 0 & 0 & 0 & 0 \end{array}} & {\arraycolsep 3pt \scriptsize \begin{array}{cccc} 0 & 0 & \llap{-}1 & 0 \\ 0 & 0 & \llap{-}1 & \llap{-}1 \\  0 & 0 & \llap{-}1 & \llap{-}1 \\ 0 & 0 & \llap{-}1 & \llap{-}1 \end{array}} & {\arraycolsep 3pt \scriptsize \begin{array}{cccc} 0 & 0 & 0 & 0 \\ 0 & 0 & 0 & 0 \\ 0 & 0 & 0 & 0 \\ 0 & 0 & 0 & 0 \end{array}} & {\arraycolsep 3pt \scriptsize \begin{array}{cccc} 0 & 0 & 0 & 0 \\ 0 & 0 & 0 & 0 \\ 0 & 0 & 0 & 0 \\ 0 & 0 & 0 & 0 \end{array}} & {\arraycolsep 3pt \scriptsize \begin{array}{cccc} 0 & 0 & 0 & 1 \\ 0 & 0 & 0 & 1\\ 0 & 0 & 0 & 1 \\ 0 & 0 & 0 & 1 \end{array}} \\ \hline
     I  & \mathcal O  &  \mathcal O  &  \mathcal O  & \mathcal O & \mathcal O \\  \hline
     \mathcal O  & I  &  \mathcal O  &  \mathcal O  & \mathcal O & \mathcal O \\  \hline
     \mathcal O  & \mathcal O  &  I  &  \mathcal O  & \mathcal O & \mathcal O \\ \hline
     \mathcal O  & \mathcal O  &  \mathcal O  &  I  & \mathcal O & \mathcal O \\ \hline
     \mathcal O  & \mathcal O  &  \mathcal O  &  \mathcal O  & I & \mathcal O
   \end{array}\right)
}
$$

Here $I$ is the $4\times 4$ identity matrix, and $\mathcal O$ is the $4\times 4$ zero matrix.

In general, the matrix $\mathcal A$ can be defined as follows. Let $L$ be the $(n+1)\times (n+1)$ lower triangular matrix whose entries are $\ell_{i,j}=1$ if $i\leq j$ and 0 otherwise. We will shift the columns of $L$ in the following way: Given a matrix $P$, define $sh(P)$ to be the matrix obtained from $P$ by removing its rightmost column and adjoining a zero-column to the left. In other words, $sh(P)$ is obtained from $P$ by shifting its columns one position to the right and inserting zeroes in the first column. When $n=3$ we can repeatedly shift the matrix $L$ to obtain:
$$
L = {\arraycolsep .2em  \tiny \begin{pmatrix} 1 & 0 & 0 & 0 \\ 1 & 1 & 0 & 0 \\ 1 & 1 & 1 & 0 \\ 1 & 1 & 1 & 1 \end{pmatrix} } \quad
  sh(L) = {\arraycolsep .2em\tiny \begin{pmatrix} 0 & 1 & 0 & 0 \\ 0 & 1 & 1 & 0 \\ 0 & 1 & 1 & 1 \\ 0 & 1 & 1 & 1 \end{pmatrix} } \quad
  sh^2(L) = {\arraycolsep .2em \tiny \begin{pmatrix} 0 & 0 & 1 & 0 \\ 0 & 0 & 1 & 1 \\ 0 & 0 & 1 & 1 \\ 0 & 0 & 1 & 1 \end{pmatrix} } \quad
  sh^3(L) = {\arraycolsep .2em \tiny \begin{pmatrix} 0 & 0 & 0 & 1 \\ 0 & 0 & 0 & 1 \\ 0 & 0 & 0 & 1 \\ 0 & 0 & 0 & 1 \end{pmatrix} }
$$

Then $\mathcal A$ is defined as a matrix made of blocks of size $(n+1)\times (n+1)$. The block in position $(i,j)$ with $1\leq i,j\leq {n+1\choose 2}$ is
$$
   \Gamma_{i,j}=\left\{\begin{array}{ll}
   (-1)^{t}sh^{t-1}(L) & \mbox{ if } i=1 \mbox{ and } j={t \choose 2} \mbox{ for some $t$.}\\ \\
   I_{(n+1)\times (n+1)} & \mbox{ if } i=j+1. \\ \\
   \mathcal O_{(n+1)\times (n+1)} & \mbox{ otherwise. } \end{array} \right.
$$
Let us show that $\mathcal A\mathbf v_{k-1} = \mathbf v_k$ for every $k\geq 1$. First, it is clear from the definition of $\mathcal A$ that the $(n+1+j)$th entry of $\mathcal A \mathbf v_{k-1}$ equals the $j$th entry of $\mathbf v_{k-1}$, which is precisely the $(n+1+j)$th entry of $\mathbf v_k$. Hence we just need to prove the equality for the first $n+1$ entries of $\mathcal A \mathbf v_{k-1}$ and $\mathbf v_k$.

We will first see that the first entry of $\mathcal A\mathbf v_{k-1}$ is $m_{k,1}$. Notice that the nonzero entries of the first row of $\mathcal A$ form an alternate sequence of $1$'s and $-1$'s placed at columns
$$
\left[{2\choose 2}-1\right](n+1)+2,\quad \left[{3\choose 2}-1\right](n+1)+3,\quad \ldots \quad \left[{n+1\choose 2}-1\right](n+1)+(n+1).
$$
The entries occupying these positions in the vector $\mathbf v_{k-1}$ are precisely
$$
m_{k-{2\choose 2},2},\quad  m_{k-{3\choose 2},3},\quad \ldots \quad  m_{k-{n+1\choose 2},n+1}.
$$
This implies that the first entry of $\mathcal A \mathbf v_{k-1}$ is $\displaystyle \sum_{j=1}^{n}{(-1)^{j-1}m_{k-{j+1\choose 2},j+1}}$ which by \autoref{C:M recurrence relation type A}  (as $m_{k,0}=0$) is precisely equal to $m_{k,1}$.

Now suppose, by induction, that $2\leq i \leq n$ and that the $(i-1)$st entry of $\mathcal A \mathbf v_{k-1}$ is $m_{k,i-1}$. We will show that the $i$th entry of $\mathcal A \mathbf v_{k-1}$ is $m_{k,i}$ and this will finish the proof.

Notice that the $i$th row of $\mathcal A$ is the sum of the $(i-1)$st row plus a vector whose nonzero entries form an alternate sequence of $1$'s and $-1$'s placed at columns
$$
 \left[{2\choose 2}-1\right](n+1)+(i+1),\quad \left[{3\choose 2}-1\right](n+1)+(i+2),\quad \ldots \quad \left[{n-i+2\choose 2}-1\right](n+1)+(n+1).
$$
(In the case $i=n+1$ all entries of this vector are zero, as the rows $n$ and $n+1$ of $\mathcal A$ are equal.)

The entries occupying these positions in $\mathbf v_{k-1}$ are
$$
   m_{k-{2\choose 2},i+1},\quad  m_{k-{3\choose 2},i+2},\quad \ldots \quad  m_{k-{n-i+2\choose 2},n+1}.
$$
Therefore the $i$th entry of $\mathcal A\mathbf v_{k-1}$ equals
$$
   \left(\mbox{$(i-1)$st row of $\mathcal A$}\right)\mathbf v_{k-1} + \sum_{j=i}^{n}{(-1)^{j-i}m_{k-{j-i+2\choose 2},j+1}}
$$
which by induction hypothesis and by \autoref{C:M recurrence relation type A} equals
$$
 m_{k,i-1}+\sum_{j=i}^{n}{(-1)^{j-i}m_{k-{j-i+2\choose 2},j+1}} = m_{k,i}
$$
as we wanted to show.
\end{proof}

We can now prove our new formula for the growth function of $A_n$, in a similar way in which growth functions of automatic groups are computed from transition matrices of finite state automata.

\begin{theorem}\label{T:typeA} Let $M$ be the Artin--Tits monoid of type $A_n$, that is the positive braid monoid on $n+1$ strands. Let $\mathcal M^{\mathcal A}_n$ be the square matrix of order $n+1$ whose entry $(i,j)$ equals $t^{j-i+1\choose 2}$ whenever $j-i+1\geq 0$, and 0 otherwise. Then
$$
    g_M(t)=|\mathcal M^{\mathcal A}_n|^{-1}
$$
\end{theorem}

\begin{proof}
It follows immediately from \autoref{L:matrix A} that $\mathcal A^k\mathbf v_0 = \mathbf v_k $ for all $k\geq 0$. Now notice that the $(n+1)$st entry of $\mathbf v_k$ is precisely $m_{k,n+1}=\alpha_k$. Hence, if we define the row vector $\mathbf v = (0 \cdots 0 \; 1 \; 0 \cdots \cdots 0)$, where the number 1 occupies the $(n+1)$st position, we obtain
$$
     \alpha_k = \mathbf v \mathcal A^k \mathbf v_0
$$
for $k\geq 0$. Therefore
$$
   g_{A_n}(t)= \sum_{k\geq 0} \alpha_k t^k =\sum_{k\geq 0} \left( \mathbf v \mathcal A^k \mathbf v_0 \right) t^k =
     \mathbf v \left(\sum_{k\geq 0} A^k t^k \right) \mathbf v_0 = \mathbf v \left(I-\mathcal At\right)^{-1} \mathbf v_0
$$
Let $\mathcal A'$ be the adjugate matrix of $I-\mathcal At$. We know that $\displaystyle (I-\mathcal At)^{-1}= \frac{\mathcal A'}{|I-\mathcal At|}$, so
$$
    g_{A_n}(t)=\frac{\mathbf v \mathcal A' \mathbf v_0}{|I-\mathcal At|}
$$

We will now show that $\mathbf v\: \mathcal A' \mathbf v_0 =1$.  To achieve this goal, notice that both $\mathbf v\: \mathcal A' \mathbf v_0$ and $|I-\mathcal At|$ are polynomials in $\mathbb Z[t]$. We will show that $|I-\mathcal At|$ has degree ${n+1\choose 2}$, and it follows directly from \autoref{C:Deligne} that $g_{A_n}(t)$ is the inverse of a polynomial of degree ${n+1\choose 2}$. Hence $\mathbf v\: \mathcal A' \mathbf v_0 $ is a constant. From $g_{A_n}(0)=\alpha_0=1$ it follows that the constant is $1$. Therefore:
$$
     g_{A_n}(t)= \frac{1}{|I-\mathcal At|}
$$

Now let us apply column operations to the matrix $I-\mathcal At$ which preserve its determinant. It is easier to explain these operations in terms of blocks. Keep in mind the case $n=3$ as a clarifying example:
$$
   |I-\mathcal At| = \left| \begin{array}{cccccc}
   I-sh(L)t & \mathcal O & sh^2(L)t & \mathcal O & \mathcal O & -sh^3(L)t \\
    -It   & I   &  \mathcal O & \mathcal O & \mathcal O & \mathcal O \\
     \mathcal O    & -It &  I & \mathcal O & \mathcal O & \mathcal O \\
     \mathcal O    & \mathcal O   & -It & I & \mathcal O & \mathcal O \\
     \mathcal O    & \mathcal O   &   \mathcal O  & -It & I & \mathcal O \\
     \mathcal O    & \mathcal O   &   \mathcal O  &  \mathcal O  & -It & I
   \end{array}\right|
$$
If we add to each (block) column the column on its right multiplied by $t$, starting from the right hand side, we obtain
$$
  |I-\mathcal At| = \left| \begin{array}{cccccc}
   I-sh(L)t+sh^2(L)t^3-sh^3(L)t^6 & \star & \star & \star & \star & \star \\
     \mathcal O   & I   &  \mathcal O & \mathcal O & \mathcal O & \mathcal O \\
     \mathcal O    & \mathcal O &  I & \mathcal O & \mathcal O & \mathcal O \\
     \mathcal O    & \mathcal O &  \mathcal O & I & \mathcal O & \mathcal O \\
     \mathcal O    & \mathcal O   &   \mathcal O  & \mathcal O & I & \mathcal O \\
     \mathcal O    & \mathcal O   &   \mathcal O  &  \mathcal O  & \mathcal O & I
   \end{array}\right|
$$
The symbol $\star$ stands for an entry whose value is not important for us, as it does not affect the determinant. Hence
$$
   |I-\mathcal At|=|I-sh(L)t+sh^2(L)t^3-sh^3(L)t^6| =
   \left|\begin{array}{rrrr}
     1 & -t & t^3 & -t^6\\
     0 & 1-t & -t+t^3 & t^3-t^6 \\
     0 & -t & 1-t+t^3 & -t+t^3-t^6 \\
     0 & -t & -t+t^3 & 1-t+t^3-t^6
   \end{array}\right|
$$
Now we substract to each row the previous one, starting from the bottom, and we obtain
$$
  |I-\mathcal At| =    \left|\begin{array}{rrrr}
     1 & -t & t^3 & -t^6 \\
     -1 & 1 & -t & t^3 \\
      0  & -1 & 1 & -t \\
      0  & 0  & -1  & 1
   \end{array}\right|
$$
Finally we point out the alternating signs, all of which can be turned positive by changing the sign of the even rows and then changing the sign of the even columns. As the number of sign changes is even, the final result is
$$
   |I-\mathcal At|= \left|\begin{array}{cccc}
     1 & t & t^3 & t^6 \\
      1 & 1 & t & t^3 \\
      0  & 1 & 1 & t \\
      0  & 0  & 1  & 1
   \end{array}\right|=|\mathcal M^{\mathcal A}_3|
$$
In the case of arbitrary $n$, exactly the same operations lead to the matrix $\mathcal M^{\mathcal A}_n$ defined in the statement.

It remains to show that $|I-\mathcal At|$ is a polynomial of degree ${n+1\choose 2}$. To achieve this, one just needs to expand the determinant $|\mathcal M^{\mathcal A}_n|$ along the first row:
$$
   |\mathcal M^{\mathcal A}_n| = 1 |\mathcal M^{\mathcal A}_{n-1}| - t |\mathcal M^{\mathcal A}_{n-2}| + t^3|\mathcal M^{\mathcal A}_{n-3}|-\cdots +(-1)^{n-1}t^{n\choose 2}|\mathcal M^{\mathcal A}_0| +(-1)^{n}t^{n+1\choose 2}.
$$
By induction in $n$, the $i$th summand ($i\leq n$) has degree ${i\choose 2}+{n-i+1\choose 2}$, which is smaller than ${n+1\choose 2}$, hence the last summand contains the leading term.
\end{proof}

The expansion of the determinant of $\mathcal M^{\mathcal A}_n$ along the first row yields precisely the previously known result by Bronfman~\cite{Bronfman}, relating the M\"obius polynomials of Artin--Tits monoids of type $A_n$, for distinct values of $n$:

\begin{theorem}\label{T:Bronfman}{\rm \cite{Bronfman}}
If the growth function of the Artin--Tits monoid of type $A_n$ is
$$
     g_{A_n}(t)=\frac{1}{H_n(t)}
$$
and we denote $H_{-1}(t)=H_0(t)=1$, then for $n\geq 1$ one has:
$$
     H_n(t)=\sum_{i=1}^{n+1} (-1)^{i-1} {t^{i\choose 2} H_{n-i}(t)}
$$
\end{theorem}

\begin{proof}
By \autoref{T:typeA} one has $H_n(t)=|\mathcal M^{\mathcal A}_n|$ for all $n\geq 1$. The formula in the statement is just the expansion of $|\mathcal M^{\mathcal A}_n|$ along the first row.
\end{proof}

To finish the section, we exhibit some examples of the matrices and the growth functions, for small values of $n$.

{\bf Examples:}
$$
  g_{A_1}(t)=\left|\begin{array}{cc} 1 & t \\ 1 & 1 \end{array}\right|^{-1},
\hspace{1cm}
  g_{A_2}(t)=\left|\begin{array}{ccc} 1 & t & t^3 \\ 1 & 1 & t \\ 0 & 1 & 1 \end{array}\right|^{-1},
$$
$$
g_{A_3}(t)=\left|\begin{array}{cccc} 1 & t & t^3 & t^6 \\ 1 & 1 & t & t^3 \\ 0 & 1 & 1 & t \\
  0 & 0 & 1 & 1 \end{array}\right|^{-1},
\quad  g_{A_4}(t)=\left|\begin{array}{ccccc} 1 & t & t^3 & t^6 & t^{10} \\ 1 & 1 & t & t^3 & t^6 \\ 0 & 1 & 1 & t & t^3 \\ 0 & 0 & 1 & 1 & t \\ 0 & 0 & 0 & 1 & 1\end{array}\right|^{-1}.
$$
In other words,
$$
g_{A_1}(t)=\frac{1}{1-t}, \quad g_{A_2}(t)=\frac{1}{1-2t+t^3}, \quad g_{A_3}(t)=\frac{1}{1-3t+t^2+2t^3-t^6},
$$
$$
g_{A_4}(t)=\frac{1}{1-4t+3t^2+3t^3-2t^4-2t^6+t^{10}}.
$$

\bigskip

\section{Artin--Tits monoids of type B and D}\label{S:B_and_D}

In the previous section we treated the case of Artin--Tits monoids of type A. In this section we will see that the same techniques can be applied to Artin--Tits monoids of type B and D. As the proofs are almost identical, we shall basically provide just the parts in which they differ.

\subsection{Counting elements in the monoid (type B)}

Let $B_n$ be the Artin--Tits monoid of type $B$ with $n$ standard generators. Its standard presentation is the following:
$$
B_n=\left\langle a_1,\ldots,a_n \; \left| \begin{array}{cl}  a_ia_j=a_ja_i, & |i-j|>1 \\ a_ia_ja_i=a_ja_ia_j, & |i-j|=1, \ i,j\neq n \\ a_{n-1}a_{n}a_{n-1}a_{n}=a_{n}a_{n-1}a_{n}a_{n-1} \end{array}   \right.\right\rangle
$$
In this case (see~\cite{Paris}), for $0\leq i\leq j\leq n$, the consecutive generators $a_i,\ldots,a_j$ either generate an Artin--Tits group of type $A_{j-i+1}$ (if $j<n$), or generate an Artin--Tits group of type $B_{j-i+1}$ (if $j=n$). In the latter case, the least common multiple of the generators is~\cite[Lemma 4.1]{Paris2}:
$$
   a_i\vee \cdots \vee a_n= (a_i\cdots a_{n-1} a_n a_{n-1}\cdots a_i)(a_{i+1}\cdots a_{n-1} a_n a_{n-1}\cdots a_{i+1})\cdots (a_{n-1} a_n a_{n-1})a_n.
$$
Therefore, one has
$$
  ||a_{i}\vee \cdots \vee a_j||=\left\{\begin{array}{ll} {j-i+2\choose 2} & \mbox{ if }\ j<n \\[.3cm] (n-i+1)^2 & \mbox{ if }\ j=n
   \end{array}\right.
$$
Also, as it happened in $A_n$, if $t\geq j+2$ the atom $a_t$ commutes with $a_{i},\ldots,a_j$, so one has
$$
   (a_{i}\vee \cdots \vee a_j)\vee a_t = (a_{i}\vee \cdots \vee a_j) a_t.
$$
Therefore we obtain the following result.

\begin{lemma}\label{L:U case B}
Let $M$ be the Artin--Tits monoid of type $B_n$. For $1\leq i \leq  j \leq n$ one has
$$
   u_{k,i,j}= \left\{\begin{array}{ll} m_{k-{j-i+2\choose 2},j+1} & \mbox{if }\ j<n, \\[.5cm] m_{k-(n-i+1)^2,n+1} & \mbox{if }\ j=n. \end{array}\right.
$$
\end{lemma}

\begin{proof}
Same proof as \autoref{L:U case A}.
\end{proof}

\begin{corollary}\label{C:M recurrence relation type B}
In the Artin--Tits monoid $B_n$, for $k\geq 0$ and $i=1,\ldots,n+1$, one has:
$$
   m_{k,i} = m_{k,i-1} + \left(\sum_{j=i}^{n-1}(-1)^{j-i}m_{k-{j-i+2\choose 2},j+1}\right)+(-1)^{n-i}m_{k-(n-i+1)^2,n+1}
$$
\end{corollary}

\begin{proof}
This is a direct consequence of \autoref{L:M in terms of U} and \autoref{L:U case B}.
\end{proof}

Using this recurrence relation we can also compute, in the case of $B_n$, a table containing the numbers $m_{k,i}$, where each row is computed from left to right, using the data in the previous ones. In \autoref{F:table B_3} we can see the first seven rows of the table corresponding to $B_3$. We recall that the rightmost column contains the number of elements of length $k$ in $B_3$.

\begin{figure}[ht]
\begin{center}
\begin{tabular}{|@{}c@{}||c|c|c|c|}
\hline $\begin{array}{ccc} & & \\ &  & i \\ & k &  \end{array}$ & 1 & 2 & 3 & 4 \\
\hline \hline 0 & 1 & 1 & 1 & 1 \\
\hline 1 & 1 & 2 & 3 & 3 \\
\hline 2 & 2 & 5 & 8 & 8 \\
\hline 3 & 4 & 12 & 20 & 20 \\
\hline 4 & 9 & 28 & 48 & 48 \\
\hline 5 & 20 & 65 & 113 & 113 \\
\hline 6 & 45 & 150 & 263 & 263 \\
\hline
\end{tabular}
\end{center}
\caption{A table containing $m_{k,i}$ for the Artin--Tits monoid $B_3$, for $k\leq 6$.}
\label{F:table B_3}
\end{figure}

\subsection{A new formula for the growth function (type B)}

We keep proceeding as in the case of Artin--Tits monoids of type $A$, although this time the defined vectors have length $n^2(n+1)$. Namely, we define for $k\geq 1$, the column vector
$$
  \mathbf v_{k-1} = \left(m_{k-1,1}\cdots m_{k-1,n+1} \: m_{k-2,1}\cdots m_{k-2,n+1} \cdots m_{k-n^2,1} \cdots m_{k-n^2,n+1}\right)^t
$$
where the $m_{i,j}$'s are defined for the Artin--Tits monoid $B_n$.

\begin{lemma}\label{L:matrix B}
There is a square matrix $\mathcal B$ with $n^2(n+1)$ rows, whose entries belong to $\{0,1,-1\}$ such that for every $k\geq 1$
$$
      \mathcal B \mathbf v_{k-1} = \mathbf v_k
$$
\end{lemma}

\begin{proof}
As in the previous section, we can indicate how the matrix $\mathcal B$ looks like, by showing it explicitly when $n=3$. As in the case of type A, let us denote (for $n=3$):
$$
sh(L) = {\arraycolsep .2em\tiny \begin{pmatrix} 0 & 1 & 0 & 0 \\ 0 & 1 & 1 & 0 \\ 0 & 1 & 1 & 1 \\ 0 & 1 & 1 & 1 \end{pmatrix} } \quad
  sh^2(L) = {\arraycolsep .2em \tiny \begin{pmatrix} 0 & 0 & 1 & 0 \\ 0 & 0 & 1 & 1 \\ 0 & 0 & 1 & 1 \\ 0 & 0 & 1 & 1 \end{pmatrix} } \quad
  sh^3(L) = {\arraycolsep .2em \tiny \begin{pmatrix} 0 & 0 & 0 & 1 \\ 0 & 0 & 0 & 1 \\ 0 & 0 & 0 & 1 \\ 0 & 0 & 0 & 1 \end{pmatrix} }
$$
This time the last column of these matrices behaves differently, so given a matrix $T$, we will denote $T'$ the matrix obtained from $T$ by replacing its last column by a column of zeroes, and $T''$ the matrix obtained from $T$ by keeping its last column and replacing all other entries by zeroes. Of course, $T'+T''=T$.

As an example, we have:
$$
  sh^2(L) = {\arraycolsep .2em \tiny \begin{pmatrix} 0 & 0 & 1 & 0 \\ 0 & 0 & 1 & 1 \\ 0 & 0 & 1 & 1 \\ 0 & 0 & 1 & 1 \end{pmatrix} }  \quad
  sh^2(L)' = {\arraycolsep .2em \tiny \begin{pmatrix} 0 & 0 & 1 & 0 \\ 0 & 0 & 1 & 0 \\ 0 & 0 & 1 & 0 \\ 0 & 0 & 1 & 0 \end{pmatrix} } \quad
  sh^2(L)'' = {\arraycolsep .2em \tiny \begin{pmatrix} 0 & 0 & 0 & 0 \\ 0 & 0 & 0 & 1 \\ 0 & 0 & 0& 1 \\ 0 & 0 & 0 & 1 \end{pmatrix} } \quad
$$

then the matrix $\mathcal B$ is the following block matrix:

\newlength{\mycolwd}% array column width
\settowidth{\mycolwd}{\small \ $-sh^3(U)'$ }

$$
\mathcal B={\arraycolsep 0pt  \small \left( \begin{array}{*{9}{C{\mycolwd}}}
   sh(L) & \mathcal O & -sh^2(L)' & -sh^2(L)'' & \mathcal O & \mathcal O & \mathcal O & \mathcal O & sh^3(L)''\\
     I  & \mathcal O & \mathcal O & \mathcal O & \mathcal O & \mathcal O & \mathcal O & \mathcal O & \mathcal O \\
     \mathcal O  & I & \mathcal O & \mathcal O & \mathcal O & \mathcal O & \mathcal O & \mathcal O & \mathcal O \\
     \mathcal O  & \mathcal O & I & \mathcal O & \mathcal O & \mathcal O & \mathcal O & \mathcal O & \mathcal O \\
     \mathcal O  & \mathcal O & \mathcal O & I & \mathcal O & \mathcal O & \mathcal O & \mathcal O & \mathcal O \\
     \mathcal O  & \mathcal O & \mathcal O & \mathcal O & I & \mathcal O & \mathcal O & \mathcal O & \mathcal O \\
     \mathcal O  & \mathcal O & \mathcal O & \mathcal O & \mathcal O & I & \mathcal O & \mathcal O & \mathcal O \\
     \mathcal O  & \mathcal O & \mathcal O & \mathcal O & \mathcal O & \mathcal O & I & \mathcal O & \mathcal O \\
     \mathcal O  & \mathcal O & \mathcal O & \mathcal O & \mathcal O & \mathcal O & \mathcal O & I & \mathcal O \\
   \end{array}\right)
}
$$
where $I$ is the $4\times 4$ identity matrix and $\mathcal O$ is the $4\times 4$ zero matrix.

For every $n\geq 2$, the $(n+1)\times (n+1)$ blocks of the matrix $\mathcal B$ are defined in the following way. The block in position $(i,j)$ with $1\leq i,j\leq n^2$ is
$$
    B_{i,j}=\Gamma_1+\Gamma_2+\Gamma_3,
$$
where
$$
\begin{array}{l}
  \Gamma_1=\left\{\begin{array}{ll}
   (-1)^{t}sh^{t-1}(L)' & \mbox{ if } i=1 \mbox{ and } j={t \choose 2} \mbox{ for some }2< t\leq n \\ \\
   \mathcal O_{(n+1)\times (n+1)} & \mbox{ otherwise. } \end{array} \right.
\\ \\
  \Gamma_2=\left\{\begin{array}{ll}
   (-1)^{t} sh^{t-1}(L)'' & \mbox{ if } i=1 \mbox{ and } j=t^2 \mbox{ for some }2\leq t\leq n \\ \\
   \mathcal O_{(n+1)\times (n+1)} & \mbox{ otherwise. } \end{array} \right.
\\ \\
  \Gamma_3=\left\{\begin{array}{ll}
   I_{(n+1)\times (n+1)} & \mbox{ if } i=j+1 \\ \\
   \mathcal O_{(n+1)\times (n+1)} & \mbox{ otherwise. } \end{array} \right.
\end{array}
$$

The proof which shows that $\mathcal B\mathbf v_{k-1} = \mathbf v_k$ for every $k\geq 1$ is analogous to that of type $A$, this time using the formula in \autoref{C:M recurrence relation type B}.
\end{proof}

\begin{theorem}\label{T:typeB}
Let $M$ be the Artin--Tits monoid of type $B_n$. Let $\mathcal M^{\mathcal B}_n$ be the square matrix of order $n+1$ whose entry $(i,n+1)$ equals $t^{(n-i+1)^2}$ for all $i$, and whose entry $(i,j)$ for $j\leq n$ equals $t^{j-i+1\choose 2}$ whenever $j-i+1\geq 0$, and 0 otherwise. Then
$$
    g_M(t)=|\mathcal M^{\mathcal B}_n|^{-1}
$$
\end{theorem}

\begin{proof}
Using the same arguments as in the proof of \autoref{T:typeA}, one has:
$$
    g_{B_n}(t)=\frac{\mathbf v \mathcal B' \mathbf v_0}{|I-\mathcal Bt|}
$$
Later we will see that $|I-\mathcal Bt|$ is a polynomial of degree $n^2$, and $g_{B_n}(t)$ is the inverse of a polynomial of degree $n^2$ by \autoref{C:Deligne}. Hence $\mathbf v \mathcal B' \mathbf v_0$ is a constant, namely 1 (as $g_{B_n}(0)=\alpha_0=1$).

Therefore:
$$
     g_{B_n}(t)= \frac{1}{|I-\mathcal Bt|}
$$

\settowidth{\mycolwd}{\small $-sh^3(L)''t\;  $}

In the case $n=3$ we have:
$$
|I-\mathcal Bt|={\arraycolsep 0pt   \small \left| \begin{array}{*{9}{C{\mycolwd}}}
  I-sh(L)t & \mathcal O & sh^2(L)'t & sh^2(L)''t & \mathcal O & \mathcal O & \mathcal O & \mathcal O & -sh^3(L)''t \\
  -It  & I & \mathcal O & \mathcal O & \mathcal O & \mathcal O & \mathcal O & \mathcal O & \mathcal O \\
    \mathcal O  & -It  & I & \mathcal O & \mathcal O & \mathcal O & \mathcal O & \mathcal O & \mathcal O \\
    \mathcal O  & \mathcal O & -It  & I & \mathcal O & \mathcal O & \mathcal O & \mathcal O & \mathcal O \\
    \mathcal O  & \mathcal O & \mathcal O & -It  & I & \mathcal O & \mathcal O & \mathcal O & \mathcal O \\
    \mathcal O  & \mathcal O & \mathcal O & \mathcal O & -It  & I & \mathcal O & \mathcal O & \mathcal O \\
    \mathcal O  & \mathcal O & \mathcal O & \mathcal O & \mathcal O & -It  & I & \mathcal O & \mathcal O \\
    \mathcal O  & \mathcal O & \mathcal O & \mathcal O & \mathcal O & \mathcal O & -It  & I & \mathcal O \\
    \mathcal O  & \mathcal O & \mathcal O & \mathcal O & \mathcal O & \mathcal O & \mathcal O & -It  & I \\
   \end{array}\right|
}
$$
As we did for type A, if we add to each (block) column the column on its right multiplied by $t$, starting from the right hand side, we obtain
$$
  |I-\mathcal Bt| = \left| \begin{array}{ccccccccc}
   T & \star & \star & \star & \star & \star & \star & \star & \star \\
     \mathcal O & I & \mathcal O & \mathcal O & \mathcal O & \mathcal O & \mathcal O & \mathcal O & \mathcal O \\
     \mathcal O & \mathcal O & I & \mathcal O & \mathcal O & \mathcal O & \mathcal O & \mathcal O & \mathcal O \\
     \mathcal O & \mathcal O & \mathcal O & I & \mathcal O & \mathcal O & \mathcal O & \mathcal O & \mathcal O \\
     \mathcal O & \mathcal O & \mathcal O & \mathcal O & I & \mathcal O & \mathcal O & \mathcal O & \mathcal O \\
     \mathcal O & \mathcal O & \mathcal O & \mathcal O & \mathcal O & I & \mathcal O & \mathcal O & \mathcal O \\
     \mathcal O & \mathcal O & \mathcal O & \mathcal O & \mathcal O & \mathcal O & I & \mathcal O & \mathcal O \\
     \mathcal O & \mathcal O & \mathcal O & \mathcal O & \mathcal O & \mathcal O & \mathcal O & I & \mathcal O \\
     \mathcal O & \mathcal O & \mathcal O & \mathcal O & \mathcal O & \mathcal O & \mathcal O & \mathcal O & I \\

   \end{array}\right|
$$
where
$$
T= I+\left(-sh(L)'t+sh^2(L)'t^3\right)+ \left(-sh(L)''t+sh^2(L)''t^4-sh^3(L)''t^9\right)
$$
Hence
$$
   |I-\mathcal Bt|=|T| =
   \left|\begin{array}{rrrr}
     1 & -t & t^3 & -t^9 \\
      0 & 1-t & -t+t^3 & t^4-t^9 \\
        0  & -t & 1-t+t^3 & -t+t^4-t^9 \\
            0  & -t & -t+t^3  & 1-t+t^4-t^9
   \end{array}\right|
$$
Notice that the exponents in the last column are of the form $k^2$, while the exponents on the other columns are of the form $k\choose 2$.

Substracting to each row the previous one, starting from the bottom, we obtain
$$
  |I-\mathcal Bt| =    \left|\begin{array}{rrrr}
     1 & -t & t^3 & -t^9 \\
      -1 & 1 & -t & t^4 \\
      0  & -1 & 1 & -t \\
      0  &  0  & -1  & 1
   \end{array}\right|= \left|\begin{array}{cccc}
     1 & t & t^3 & t^9 \\
      1 & 1 & t & t^4 \\
      0  & 1 & 1 & t \\
      0  &  0  & 1  & 1
   \end{array}\right|=|\mathcal M^{\mathcal B}_3|
$$
In the case of arbitrary $n$, exactly the same operations lead to the matrix $\mathcal M^{\mathcal B}_n$ defined in the statement of \autoref{T:typeB}. Hence $|I-\mathcal Bt|=|\mathcal M^{\mathcal B}_n|$.

Notice that $M^{\mathcal B}_n$ equals $M^{\mathcal A}_n$ except for the last column. Hence, if we expand the determinant of $\mathcal M^{\mathcal B}_n$ along the last column, we see that we can express $|\mathcal M^{\mathcal B}_n|$ in terms of $|\mathcal M^{\mathcal A}_t|$ for $t<n$. Namely:
$$
   |\mathcal M_n^{\mathcal B}| = 1 |\mathcal M^{\mathcal A}_{n-1}| - t |\mathcal M^{\mathcal A}_{n-2}| + t^4|\mathcal M^{\mathcal A}_{n-3}|-\cdots +(-1)^{n-1}t^{(n-1)^2}|\mathcal M^{\mathcal A}_0| +(-1)^{n}t^{n^2}.
$$
As $|\mathcal M^{\mathcal A}_t|$ is a polynomial of degree $t\choose 2$, it follows that $|\mathcal M_n^{\mathcal B}|$ is a polynomial of degree $n^2$. This finishes the proof.
\end{proof}

The above expansion of the determinant of $\mathcal M^{\mathcal B}_n$ along the last column, yields a formula relating the M\"obious polynomials of $B_n$ to the M\"obius polynomials of $A_m$ for $m<n$.

\begin{corollary}
If the growth functions of the Artin--Tits monoids of types $A_n$ and $B_n$ are
$$
     g_{A_n}=\frac{1}{H_n(t)}, \qquad g_{B_n}=\frac{1}{F_n(t)}
$$
and we denote $H_{-1}(t)=H_0(t)=1$, then for $n\geq 1$ one has:
$$
     F_n(t)=\sum_{i=0}^{n} (-1)^{i} {t^{i^2} H_{n-1-i}(t)}
$$
\end{corollary}

\begin{proof}
By \autoref{T:typeB} one has $F_n(t)=|\mathcal M^{\mathcal B}_n|$ for all $n\geq 1$. The formula in the statement is just the expansion of $|\mathcal M^{\mathcal B}_n|$ along the last column.
\end{proof}

But we can also expand the determinant $|\mathcal M^{\mathcal B}_n|$ along the first row, yielding a formula relating the growth function of $B_n$ to the growth functions of $B_m$ for $m<n$.

\begin{corollary}
If the growth functions of the Artin--Tits monoid of type $B_n$ is
$$
     g_{B_n}=\frac{1}{F_n(t)}
$$
and we denote $F_{-1}(t)=F_0(t)=1$, then for $n\geq 1$ one has:
$$
     F_n(t)=\left(\sum_{i=1}^{n} (-1)^{i-1} {t^{i\choose 2} F_{n-i}(t)}\right)+(-1)^{n}t^{n^2}
$$
\end{corollary}

\begin{proof}
By \autoref{T:typeB} one has $F_n(t)=|\mathcal M^{\mathcal B}_n|$ for all $n\geq 1$. Now expand this determinant along the first row.
\end{proof}

As we did with the monoid of type A, we finish the section with some examples.

{\bf Examples:}
$$
  g_{B_1}(t)=\left|\begin{array}{cc} 1 & t \\ 1 & 1 \end{array}\right|^{-1},
\hspace{1cm}
  g_{B_2}(t)=\left|\begin{array}{ccc} 1 & t & t^4 \\ 1 & 1 & t \\ 0 & 1 & 1 \end{array}\right|^{-1},
$$
$$
g_{B_3}(t)=\left|\begin{array}{cccc} 1 & t & t^3 & t^9 \\ 1 & 1 & t & t^4 \\ 0 & 1 & 1 & t \\
  0 & 0 & 1 & 1 \end{array}\right|^{-1},
\quad  g_{B_4}(t)=\left|\begin{array}{ccccc} 1 & t & t^3 & t^6 & t^{16} \\ 1 & 1 & t & t^3 & t^9 \\ 0 & 1 & 1 & t & t^4 \\ 0 & 0 & 1 & 1 & t \\ 0 & 0 & 0 & 1 & 1 \end{array}\right|^{-1}.
$$
In other words,
$$
g_{B_1}(t)=\frac{1}{1-t}, \quad g_{B_2}(t)=\frac{1}{1-2t+t^4}, \quad g_{B_3}(t)=\frac{1}{1-3t+t^2+t^3+t^4-t^9},
$$
$$
g_{B_4}(t)=\frac{1}{1-4t+3t^2+2t^3-t^5-t^6-t^9+t^{16}}.
$$

\bigskip

\subsection{Counting elements in the monoid (type D)}

The standard presentation of the Artin--Tits monoid of type $D$ with $n\geq 4$ generators is the following:
$$
D_n=\left\langle a_1,\ldots,a_n \; \left| \begin{array}{cl}
a_ia_j=a_ja_i, & |i-j|>1,\ \{i,j\}\neq \{n-2,n\} \\ a_ia_ja_i=a_ja_ia_j, & |i-j|=1, \ \{i,j\}\neq \{n-1,n\} \\
a_{n-2}a_n a_{n-2}=a_n a_{n-2} a_n \\ a_{n-1}a_n=a_n a_{n-1}
\end{array}   \right.\right\rangle
$$

In this monoid (see~\cite{Paris}), the submonoid generated by the consecutive standard generators $a_i,\ldots,a_j$ is an Artin--Tits monoid of type $D$ if $j=n$ and $i\leq n-3$, it is $\mathbb Z^2$ if $j=n$ and $i=n-1$, and it is an Artin--Tits monoid of type $A$ otherwise. In the first case, the least common multiple of the generators is~\cite[Lemma 5.1]{Paris2}:
$$
 a_i\vee \cdots \vee a_n = \prod_{j=i}^{n-1}{(a_j a_{j+1}\cdots a_{n-2})(a_{n-1}a_n)(a_{n-2}a_{n-3}\cdots a_j)}.
$$
This word has length $(n-i+1)(n-i)$. We remark that if $i=n-2$ and $j=n$, the least common multiple of $a_{n-2},a_{n-1},a_n$ (which generate a group of type $A_3$) is an element of length $6$, and this number is also equal to $(n-i+1)(n-i)$. And if $i=n-1$ and $j=n$, the generators $a_{n-1}$ and $a_n$ commute, so their least common multiple is $a_{n-1}a_n$ which has length 2, that is, equal to $(n-i+1)(n-i)$. Therefore:
$$
  ||a_{i}\vee \cdots \vee a_j|| =\left\{\begin{array}{ll}
    (n-i+1)(n-i) & \mbox{ if }\ j=n \mbox{ and } i< n  \\[.1cm]
    {j-i+2 \choose 2} & \mbox{ otherwise. }
  \end{array} \right.
$$

On the other hand, if $t\geq j+2$, the atom $a_t$ commutes with $a_{i},\ldots,a_j$, except when $t=n$ and $j=n-2$. In this latter case, the generators $a_i,a_{i+1},\ldots,a_{n-2},a_n$ behave as consecutive generators in a monoid of type $A$. Hence, when $t\geq j+2$:
$$
   (a_{i}\vee \cdots \vee a_j)\vee a_t =\left\{ \begin{array}{ll}
     (a_{i}\vee \cdots \vee a_j)a_t & \mbox{ if }\ j<n-2 \\
     (a_{i}\vee \cdots \vee a_j)a_t a_j a_{j-1}\cdots a_i & \mbox{ if }\ j=n-2 \mbox{ (hence $t=n$)}
   \end{array}\right.
$$
Therefore, in $D_n$ the situation is the following.

\begin{lemma}\label{L:U case D}
Let $M$ be the Artin--Tits monoid of type $D_n$ ($n\geq 4$). For $1\leq i \leq  j \leq n$ one has
$$
   u_{k,i,j}= \left\{\begin{array}{ll}
    m_{k-(n-i+1)(n-i),n+1} & \mbox{ if }\ j=n \mbox{ and } i< n  \\[.2cm]
    m_{k-{n-i\choose 2},n}-m_{k-{n-i+1\choose 2},n+1} & \mbox{ if }\ j=n-2, \\[.2cm]
    m_{k-{j-i+2\choose 2},j+1} & \mbox{ otherwise.}
    \end{array}\right.
$$
\end{lemma}

\begin{proof}
If $j=n$ and $i<n$ we have
$$
u_{k,i,j}= \left| \{b\in M_k;\quad a_i\vee \cdots \vee a_n\preccurlyeq b\} \right| = |M_{k-(n-i+1)(n-i)}|=m_{k-(n-i+1)(n-i),n+1}.
$$

Suppose now that $j=n-2$. In this case:
\begin{eqnarray*}
 U_{k}^{(i,n-2)} & = & \{b\in M_k;\quad  a_{i}\vee \cdots \vee a_{n-2} \preccurlyeq b, \quad a_n\not\preccurlyeq b  \} \\
 & = & \{(a_{i}\vee \cdots \vee a_{n-2})c \in M_k;\quad a_n\not\preccurlyeq (a_{i}\vee \cdots \vee a_{n-1})c \}
\end{eqnarray*}
But $a_n\preccurlyeq (a_{i}\vee \cdots \vee a_{n-2})c$ if and only if $(a_{i}\vee \cdots \vee a_{n-2})\vee a_n\preccurlyeq (a_i\vee \cdots \vee a_{n-2})c$, which occurs if and only if $(a_{i}\vee \cdots \vee a_{n-2})\: a_n a_{n-2} a_{n-3}\cdots a_i\preccurlyeq (a_{i}\vee \cdots \vee a_{n-2})c$, which is equivalent to $a_n a_{n-2} a_{n-3}\cdots a_i\preccurlyeq c$. Hence
$$
 U_{k}^{(i,n-2)} = \{(a_{i}\vee \cdots \vee a_{n-2})c \in M_k;\ a_n a_{n-2} a_{n-3}\cdots a_i\not\preccurlyeq c \}.
$$
Therefore
$$
u_{k,i,n-2}=\left|U_{k}^{(i,n-2)} \right|= \left|\{c \in M_{k-{n-i\choose 2}};\ a_n a_{n-2} a_{n-3}\cdots a_i\not\preccurlyeq c \}\right|.
$$
Hence, we must remove from the elements in $M_{k-{n-i\choose 2}}$ those having $a_n a_{n-2} a_{n-3}\cdots a_i$ as prefix. As $D_n$ is a cancellative monoid, there are exactly $\left| M_{k-{n-i\choose 2}-(n-i)}\right|=\left|M_{k-{n-i+1\choose 2}}\right|$ elements to remove. Therefore
$$
u_{k,i,n-2}=\left| M_{k-{n-i\choose 2}}\right| - \left| M_{k-{n-i+1\choose 2}}\right| = m_{k-{n-i\choose 2},n}-m_{k-{n-i+1\choose 2},n+1}
$$
as we wanted to show. The reason of using either $n$ or $n+1$ as second subindex will become clear later.

Finally, in all the remaining cases (either $i=j=n$ or $j\notin \{n-2,n\})$, we have:
$$
  U_{k}^{(i,j)}= \{b\in M_k;\quad  a_{i}\vee \cdots \vee a_j \preccurlyeq b, \quad a_{j+2},\ldots,a_{n}\not\preccurlyeq b \} = m_{k-{j-i+2\choose 2},j+1}
$$
in the same way as in \autoref{L:U case A}.
\end{proof}

\begin{corollary}\label{C:M recurrence relation type D}
In the Artin--Tits monoid $D_n$ ($n\geq 4$), for $k\geq 0$ and $i=1,\ldots,n-2$, one has:
\begin{eqnarray*}
   m_{k,i} = m_{k,i-1} & + & \left(\sum_{j=i}^{n-3}(-1)^{j-i}m_{k-{j-i+2\choose 2},j+1}\right) \\ & + & (-1)^{n-2-i}\left(m_{k-{n-i \choose 2},n}-2\: m_{k-{n-i+1 \choose 2},n+1}+m_{k-(n-i+1)(n-i),n+1}\right),
\end{eqnarray*}
and also:
$$
  m_{k,n-1}= m_{k,n-2} + m_{k-1,n} - m_{k-2,n+1}
$$
$$
  m_{k,n} = m_{k,n-1} + m_{k-1,n+1},
$$
$$
  m_{k,n+1}=m_{k,n}.
$$
\end{corollary}

\begin{proof}
This is a direct consequence of \autoref{L:M in terms of U} and \autoref{L:U case D}.
\end{proof}

The first formula in \autoref{C:M recurrence relation type D} has an interesting property: It involves no expression of the form $m_{t,n-1}$ for any $t$. Hence, the formulae can be simplified if we remove the column corresponding to $i=n-1$ from the table containing the $m_{k,i}$'s. For that purpose, we define:

\begin{definition}
Let $M$ be the Artin--Tits monoid of type $D_n$. For $k\geq 0$ we define:
$$
 d_{k,i}=\left\{ \begin{array}{ll} m_{k,i} & \ \mbox{ if }\ i\in \{1,\ldots, n-2\}, \\
                                   m_{k,i+1} &  \ \mbox{ if }\ i\in\{ n-1, n\}.
 \end{array}\right.
$$
\end{definition}

We have then defined the numbers $d_{k,1},\ldots,d_{k,n}$, such that $d_{k,n-1}=d_{k,n}=|M_k|=\alpha_k$.  With this definition, the recurrence relation in~\autoref{C:M recurrence relation type D} can be rewritten as follows:

\begin{corollary}\label{C:M second_recurrence relation type D}
In the Artin--Tits monoid $D_n$, for $k\geq 0$ and $i=1,\ldots,n$ one has:
$$
    d_{k,i}=d_{k,i-1}+\left( \sum_{j=i}^{n-1} (-1)^{j-i} d_{k-{j-i+2\choose 2},j+1} \right) +
     (-1)^{n-i-1}\left( d_{k-{n-i+1\choose 2},n}- d_{k-(n-i+1)(n-i),n} \right)
$$
\end{corollary}

\begin{proof}
For $i=1,\ldots, n-2$, this is a direct consequence of \autoref{C:M recurrence relation type D}.

For $i=n-1$, we have (using \autoref{C:M recurrence relation type D} and the fact that $m_{k-1,n}=m_{k-1,n+1}$):
\begin{eqnarray*}
  d_{k,n-1} = m_{k,n} & = & m_{k,n-1}+m_{k-1,n+1} \\
  & = & (m_{k,n-2} + m_{k-1,n} - m_{k-2,n+1}) +m_{k-1,n+1} \\
  & = & m_{k,n-2} + 2\: m_{k-1,n+1} - m_{k-2,n+1} \\
  & = & d_{k,n-2} + 2\: d_{k-1,n} - d_{k-2,n}
\end{eqnarray*}
which satisfies the statement.

Finally, if $i=n$:
$$
d_{k,n}=m_{k,n+1}=m_{k,n}=d_{k,n-1}= d_{k,n-1}- (d_{k,n}-d_{k,n})
$$
so the result also holds in this case.
\end{proof}

Thanks to \autoref{C:M second_recurrence relation type D} we can compute a table containing the numbers $d_{k,i}$ for $k\geq 0$ and $i=1,\ldots,n$. Notice that the rightmost column contains the numbers $d_{k,n}$ for $k\geq 0$, which correspond to the number of elements of length $k$, with respect to the standard generators of $D_n$.

\begin{figure}[ht]
\begin{center}
\begin{tabular}{|@{}c@{}||c|c|c|c|}
\hline $\begin{array}{ccc} & & \\ &  & i \\ & k &  \end{array}$ & 1 & 2 & 3 & 4 \\
\hline \hline 0 & 1 & 1 & 1 & 1 \\
\hline 1 & 1 & 2 & 4 & 4 \\
\hline 2 & 2 & 6 & 13 & 13 \\
\hline 3 & 5 & 16 & 38 & 38 \\
\hline 4 & 12 & 42 & 105 & 105 \\
\hline 5 & 29 & 108 & 280 & 280 \\
\hline 6 & 72 & 277 & 732 & 732 \\
\hline
\end{tabular}
\end{center}
\caption{A table containing the numbers $d_{k,i}$ for the Artin--Tits monoid $D_4$, for $k\leq 6$.}
\label{F:table D_4}
\end{figure}

%\subsection{Counting elements in an Artin--Tits monoid of type E}
%
%There are three Artin--Tits monoids of type $E$, namely $E_6$, $E_7$ and $E_8$. Although their growth series are known, we will compute them using the arguments in this paper.
%
%We remark that in an Artin--Tits monoid of type $E$ one has the following data:
%$$
%\begin{tabular}{|c||c|c|c|c|c|c|c|c|}
%\hline $i$ & 1 & 2 & 3 & 4 & 5 & 6 & 7 & 8 \\
%\hline $||a_1\vee \cdots \vee a_i||$ & 1 & 3 & 4 & 10 & 20 & 36 & 63 & 120 \\
%\hline
%\end{tabular}
%$$
%Also $||a_2\vee\cdots \vee a_i||=(i-1)(i-2)$ for $i=3,\ldots,8$, and $||a_{j+1}\vee\cdots \vee a_i||={i-j+1\choose 2}$ for $2\leq j<i\leq 8$.
%
%On the other hand, for $1\leq t < j <i\leq 8$ one has
%$$
%    (a_{j+1}\vee \cdots \vee a_i)\vee a_t =\left\{\begin{array}{ll}
%      (a_4\vee \cdots \vee a_i)\:a_2a_4\cdots a_i, & \ t=2, \ j=3 \\
%      (a_{j+1}\vee \cdots \vee a_i)\:a_t & \mbox{ otherwise}
%    \end{array}\right.
%$$
%Hence for type $E$ we have the following.
%
%\begin{lemma}\label{L:U case E}
%Let $M$ be the Artin--Tits monoid of type $E_n$ with $n=6,7,8$. For $1\leq j \leq  i \leq n$ one has
%$$
%   u_{k,i-j,i}= \left\{\begin{array}{ll}
%    m_{k-i(i-1),0} & \mbox{ if }\ i=j>1 \\[.2cm]
%    m_{k-{i\choose 2},0} & \mbox{ if }\ i-j=1, \\[.2cm]
%    m_{k-{i-1\choose 2},1}-m_{k-{i\choose 2},0} & \mbox{ if }\ i-j=2, \\[.2cm]
%    m_{k-{j+1\choose 2},i-j} & \mbox{ otherwise.}
%    \end{array}\right.
%$$
%\end{lemma}

\subsection{A new formula for the growth function (type D)}

The linear recurrence obtained in the case of Artin--Tits monoids of type $D$ forces us to define vectors with $n^2(n-1)$ entries. Namely, we define for $k\geq 1$ the column vector
$$
  \mathbf v_{k-1} = \left(d_{k-1,1}\cdots d_{k-1,n} \: d_{k-2,1}\cdots d_{k-2,n} \cdots d_{k-n(n-1),1} \cdots d_{k-n(n-1),n}\right)^t
$$

\begin{lemma}\label{L:matrix D}
There is a square matrix $\mathcal D$ with $n^2(n-1)$ rows, whose entries belong to $\{0,1,-1\}$ such that for every $k\geq 1$
$$
      \mathcal D \mathbf v_{k-1} = \mathbf v_k
$$
\end{lemma}

\begin{proof}
This proof is analogous to those of type $A$ and $B$. One can see the matrix $\mathcal D$ as a squared block matrix, with $n(n-1)$ rows, each block made of $n\times n$ matrices. In the case $n=4$, the first row of blocks is the following:
$$
\arraycolsep 2pt
\scriptsize
\begin{array}{|rrrr|rrrr|rrrr|rrrr|rrrr|rrrr|rrrr|rrrr|rrrr|rrrr|rrrr|rrrr|}
\hline
0 & 1 & 0 & 0 &    0 & 0 & 0 & 0 &   0 & 0 & \llap{-}1 & 0 &
0 & 0 & 0 & 0 &    0 & 0 & 0 & 0 &   0 & 0 & 0 & 2 &
0 & 0 & 0 & 0 &    0 & 0 & 0 & 0 &   0 & 0 & 0 & 0 &
0 & 0 & 0 & 0 &    0 & 0 & 0 & 0 &   0 & 0 & 0 & \llap{-}1
\\
0 & 1 & 1 & 0 &    0 & 0 & 0 & 0 &   0 & 0 & \llap{-}1 & \llap{-}2 &
0 & 0 & 0 & 0 &    0 & 0 & 0 & 0 &   0 & 0 & 0 & 3 &
0 & 0 & 0 & 0 &    0 & 0 & 0 & 0 &   0 & 0 & 0 & 0 &
0 & 0 & 0 & 0 &    0 & 0 & 0 & 0 &   0 & 0 & 0 & \llap{-}1
\\
0 & 1 & 1 & 2 &    0 & 0 & 0 & \llap{-}1 &   0 & 0 & \llap{-}1 & \llap{-}2 &
0 & 0 & 0 & 0 &    0 & 0 & 0 & 0 &   0 & 0 & 0 & 3 &
0 & 0 & 0 & 0 &    0 & 0 & 0 & 0 &   0 & 0 & 0 & 0 &
0 & 0 & 0 & 0 &    0 & 0 & 0 & 0 &   0 & 0 & 0 & \llap{-}1
\\
0 & 1 & 1 & 2 &    0 & 0 & 0 & \llap{-}1 &   0 & 0 & \llap{-}1 & \llap{-}2 &
0 & 0 & 0 & 0 &    0 & 0 & 0 & 0 &   0 & 0 & 0 & 3 &
0 & 0 & 0 & 0 &    0 & 0 & 0 & 0 &   0 & 0 & 0 & 0 &
0 & 0 & 0 & 0 &    0 & 0 & 0 & 0 &   0 & 0 & 0 & \llap{-}1
 \\ \hline
\end{array}
$$
And the remaining rows are made of identity and zero matrices, with the identities placed at the subdiagonal.

Recall the definition of $sh^k(L)$ for $k>0$. As for type $B$, the last column of these matrices behaves differently than in the case of type $A$, so given a matrix $T$, we denote $T'$ the matrix obtained from $T$ by replacing its last column by a column of zeroes, and $T''$ the matrix obtained from $T$ by keeping its last column and replacing all other entries by zeroes. Then, for every $n\geq 2$, the $n\times n$ blocks of the matrix $\mathcal D$ are defined in the following way.

The block in position $(i,j)$ with $1\leq i,j\leq n(n-1)$ is
$$
   D_{i,j}= \Gamma_1+\Gamma_2+\Gamma_3+\Gamma_4
$$
where
$$
\begin{array}{l}
  \Gamma_1=\left\{\begin{array}{ll}
   (-1)^{t}sh^{t-1}(L)' & \mbox{ if } i=1 \mbox{ and } j={t \choose 2} \mbox{ for some }2\leq t < n \\ \\
   \mathcal O_{n\times n} & \mbox{ otherwise. } \end{array} \right.
\\ \\
  \Gamma_2=\left\{\begin{array}{ll}
   (-1)^{t}\: 2\: sh^{t-1}(L)'' & \mbox{ if } i=1 \mbox{ and } j={t \choose 2} \mbox{ for some }2\leq t\leq n \\ \\
   \mathcal O_{n\times n} & \mbox{ otherwise. } \end{array} \right.
\\ \\
  \Gamma_3=\left\{\begin{array}{ll}
   (-1)^{t-1}\: sh^{t-1}(L)'' & \mbox{ if } i=1 \mbox{ and } j={t(t-1)} \mbox{ for some }2\leq t\leq n \\ \\
   \mathcal O_{n\times n} & \mbox{ otherwise. } \end{array} \right.
\\ \\
  \Gamma_4=\left\{\begin{array}{ll}
   I_{n\times n} & \mbox{ if } i=j+1 \\ \\
   \mathcal O_{n\times n} & \mbox{ otherwise. } \end{array} \right.
\end{array}
$$
This description follows directly from the formula in \autoref{C:M second_recurrence relation type D}.
\end{proof}

\begin{theorem}\label{T:typeD}
Let $M$ be the Artin--Tits monoid of type $D_n$, for $n\geq 4$. Let $\mathcal M^{\mathcal D}_n$ be the square matrix of order $n$ whose entry $(i,n)$ equals $2t^{n-i+1\choose 2}-t^{(n-i+1)(n-i)}$ for all $i$, and whose entry $(i,j)$ for $j<n$ equals $t^{j-i+1\choose 2}$ whenever $j-i+1\geq 0$, and 0 otherwise. Then
$$
    g_M(t)=|\mathcal M^{\mathcal D}_n|^{-1}
$$
\end{theorem}

\begin{proof}
Following the lines of \autoref{T:typeA} and \autoref{T:typeB}, one has:
$$
    g_{D_n}(t)=\frac{\mathbf v \mathcal D' \mathbf v_0}{|I-\mathcal Dt|}
$$
Later we will see that $|I-\mathcal Dt|$ is a polynomial of degree $n(n-1)$, and $g_{D_n}(t)$ is the inverse of a polynomial of degree $n(n-1)$ by \autoref{C:Deligne}. Hence $\mathbf v \mathcal D' \mathbf v_0$ is a constant, namely 1.

Therefore:
$$
     g_{D_n}(t)= \frac{1}{|I-\mathcal Dt|}
$$

As we did for types A and B, if we add to each (block) column the column on its right multiplied by $t$, starting from the right hand side, we obtain a matrix which for $n=4$ looks like the following:
$$
  |I-\mathcal Dt| = \left| \begin{array}{cccccccccccc}
   T & \star & \star & \star & \star & \star & \star & \star & \star & \star & \star & \star\\
     \mathcal O & I & \mathcal O & \mathcal O & \mathcal O & \mathcal O & \mathcal O & \mathcal O & \mathcal O & \mathcal O & \mathcal O & \mathcal O \\
     \mathcal O & \mathcal O & I & \mathcal O & \mathcal O & \mathcal O & \mathcal O & \mathcal O & \mathcal O & \mathcal O & \mathcal O & \mathcal O \\
     \mathcal O & \mathcal O & \mathcal O & I & \mathcal O & \mathcal O & \mathcal O & \mathcal O & \mathcal O & \mathcal O & \mathcal O & \mathcal O \\
     \mathcal O & \mathcal O & \mathcal O & \mathcal O & I & \mathcal O & \mathcal O & \mathcal O & \mathcal O & \mathcal O & \mathcal O & \mathcal O \\
     \mathcal O & \mathcal O & \mathcal O & \mathcal O & \mathcal O & I & \mathcal O & \mathcal O & \mathcal O & \mathcal O & \mathcal O & \mathcal O \\
     \mathcal O & \mathcal O & \mathcal O & \mathcal O & \mathcal O & \mathcal O & I & \mathcal O & \mathcal O & \mathcal O & \mathcal O & \mathcal O \\
     \mathcal O & \mathcal O & \mathcal O & \mathcal O & \mathcal O & \mathcal O & \mathcal O & I & \mathcal O & \mathcal O & \mathcal O & \mathcal O \\
     \mathcal O & \mathcal O & \mathcal O & \mathcal O & \mathcal O & \mathcal O & \mathcal O & \mathcal O & I & \mathcal O & \mathcal O & \mathcal O \\
     \mathcal O & \mathcal O & \mathcal O & \mathcal O & \mathcal O & \mathcal O & \mathcal O & \mathcal O & \mathcal O & I & \mathcal O & \mathcal O \\
     \mathcal O & \mathcal O & \mathcal O & \mathcal O & \mathcal O & \mathcal O & \mathcal O & \mathcal O & \mathcal O & \mathcal O & I & \mathcal O \\
     \mathcal O & \mathcal O & \mathcal O & \mathcal O & \mathcal O & \mathcal O & \mathcal O & \mathcal O & \mathcal O & \mathcal O & \mathcal O & I \\
   \end{array}\right|
$$
Hence
$$
   |I-\mathcal Dt|=|T| =
   \left|\begin{array}{rrrr}
     1 & -t & t^3 & -2t^6+t^{12} \\
     0 & 1-t & -t+t^3 & 2t^3-3t^6+t^{12} \\
     0 & -t & 1-t+t^3 & -2t+t^2+2t^3-3t^6+t^{12} \\
     0 & -t & -t+t^3 & 1-2t+t^2+2t^3-3t^6+t^{12}
   \end{array}\right|
$$

Substracting to each row the previous one, starting from the bottom, we obtain
$$
  |I-\mathcal Dt| =    \left|\begin{array}{rrrr}
     1 & -t & t^3 & -2t^6+t^{12} \\
     -1 & 1 & -t  & 2t^3-t^6 \\
     0 & -1 & 1 &  -2t + t^2 \\
     0 & 0 & -1 & 1
   \end{array}\right|= \left|\begin{array}{cccc}
     1 & t & t^3 & 2t^6-t^{12} \\
     1 & 1 & t  & 2t^3-t^6 \\
     0 & 1 & 1 &  2t - t^2 \\
     0 & 0 & 1 & 1
   \end{array}\right|=|\mathcal M^{\mathcal D}_4|
$$
In the case of arbitrary $n$, exactly the same operations lead to the matrix $\mathcal M^{\mathcal D}_n$ defined in the statement of \autoref{T:typeD}. Hence $|I-\mathcal Dt|=|\mathcal M^{\mathcal D}_n|$.

Notice that, as for type B, the matrix $\mathcal M^{\mathcal D}_n$ equals $\mathcal M^{\mathcal A}_{n-1}$ except for the last column. Hence, if we expand the determinant of $\mathcal M^{\mathcal D}_n$ along the last column, we see that we can express $|\mathcal M^{\mathcal D}_n|$ in terms of $|\mathcal M^{\mathcal A}_t|$ for $t<n-1$. Namely:
$$
   |\mathcal M_n^{\mathcal D}| = 1 |\mathcal M^{\mathcal A}_{n-2}| - (2t-t^2) |\mathcal M^{\mathcal A}_{n-3}| + \cdots +(-1)^n (2t^{n-1\choose 2}-t^{(n-1)(n-2)})|\mathcal M^{\mathcal A}_0| +(-1)^{n-1}(2t^{n\choose 2}-t^{n(n-1)}).
$$
As $|\mathcal M^{\mathcal A}_t|$ is a polynomial of degree $t\choose 2$, it follows that $|\mathcal M_n^{\mathcal D}|$ is a polynomial of degree $n(n-1)$. This finishes the proof.
\end{proof}

From this formula it is easy to relate the growth functions of the monoid $D_n$ with the growth functions of monoids of type $A_n$.

\begin{corollary}
If the growth functions of the Artin--Tits monoids of types $A_n$ and $D_n$ are
$$
     g_{A_n}=\frac{1}{H_n(t)}, \qquad g_{D_n}=\frac{1}{G_n(t)}
$$
and we denote $H_{-1}(t)=H_0(t)=1$, then for $n>1$ one has:
$$
     G_n(t)=\sum_{i=1}^{n} (-1)^{i-1} {\left(2t^{i\choose 2}-t^{i(i-1)}\right) H_{n-i-1}(t)}
$$
\end{corollary}

\begin{proof}
By \autoref{T:typeD} one has $G_n(t)=|\mathcal M^{\mathcal D}_n|$ for all $n> 1$. The formula in the statement is just the expansion of $|\mathcal M^{\mathcal D}_n|$ along the last column.
\end{proof}

But we can also expand the determinant $|\mathcal M^{\mathcal D}_n|$ along the first row, yielding a formula relating the growth function of $D_n$ to the growth functions of $D_m$ for $m<n$. To simplify the statement, although the monoid $D_n$ is defined for $n\geq 4$, we will naturally assume that $D_2=\mathbb Z \times \mathbb Z$ and $D_3=A_3$.

\begin{corollary}\label{C:Mobius_pols_type_D}
If the growth functions of the Artin--Tits monoid of type $D_n$, for $n\geq 2$, is
$$
     g_{D_n}=\frac{1}{G_n(t)}
$$
and we denote $G_1(t)=1$, then for $n\geq 2$ one has:
$$
     G_n(t)=\left(\sum_{i=1}^{n-1} (-1)^{i-1} {t^{i\choose 2} G_{n-i}(t)}\right)+(-1)^{n-1}\left(2t^{n\choose 2}-t^{n(n-1)}\right)
$$
\end{corollary}

\begin{proof}
By \autoref{T:typeD} one has $G_n(t)=|\mathcal M^{\mathcal D}_n|$ for all $n\geq 4$. Now expand this determinant along the first row.
\end{proof}

%\begin{remark}
%If one defines the monoids $D_2$ and $D_3$ in the natural way, one has $D_2\simeq \mathbb Z \times \mathbb Z$ and $D_3\simeq A_3$. Then, the formula in \autoref{C:Mobius_pols_type_D} holds for every $n\geq 2$, taking $G_{0}(t)=G_1(t)=1$.
%\end{remark}

As in the previous cases, here are some examples to illustrate the computations.

{\bf Examples:}
$$
  g_{D_2}(t)=\left|\begin{array}{cc} 1 & 2t-t^2 \\ 1 & 1 \end{array}\right|^{-1},
\hspace{1cm}
  g_{D_3}(t)=\left|\begin{array}{ccc} 1 & t & 2t^3-t^6 \\ 1 & 1 & 2t-t^2 \\ 0 & 1 & 1 \end{array}\right|^{-1},
$$
$$
g_{D_4}(t)=\left|\begin{array}{cccc} 1 & t & t^3 & 2t^6-t^{12} \\  1 & 1 & t & 2t^3-t^6\\ 0 & 1 & 1 & 2t-t^2 \\
  0 & 0 & 1 & 1 \end{array}\right|^{-1},
\quad  g_{D_5}(t)=\left|\begin{array}{ccccc} 1 & t & t^3 & t^6 & 2t^{10}-t^{20} \\ 1 & 1 & t & t^3 & 2t^6-t^{12} \\ 0 & 1 & 1 & t & 2t^3-t^6 \\ 0 & 0 & 1 & 1 & 2t-t^2  \\ 0 & 0 & 0 & 1 & 1 \end{array}\right|^{-1}.
$$
In other words,
$$
g_{D_2}(t)=\frac{1}{1-2t+t^2}, \quad g_{D_3}(t)=\frac{1}{1-3t+t^2+2t^3-t^6}, \quad g_{D_4}(t)=\frac{1}{1-4t+3t^2+2t^3-3t^6+t^{12}},
$$
$$
g_{D_5}(t)=\frac{1}{1-5t+6t^2+2t^3-4t^4+t^5-4t^6+t^7+2t^{10}+t^{12}-t^{20}}.
$$

\section{Growth rates}\label{S:Growth_rates}

One of the big advantages of the way we use for counting the number of elements in a monoid, is that one can easily extract some conclusions on its growth rate. We recall that the exponential growth rate of a monoid, with respect of a given set of generators is equal to
$$
    \rho_M=\lim_{k\to \infty}{\sqrt[k]{\left|M_k\right|}}
$$

In the case of Artin--Tits monoids of type $A_n$, $B_n$ and $D_n$, this growth rate has been extensively studied by several authors~\cite{Saito_2009,VNB,Juge}. In~\cite{Juge}, it is shown that the M\"obius polynomial for $M$ has a unique root of smallest modulus, which is a real number. Its inverse is precisely $\rho_M$.

If one considers the family of monoids $\{A_n\}_{n\geq 1}$, it is a natural question, posed by several authors, to find the limit:
$$
   \rho=\lim_{n\to \infty}{\rho_{A_n}}.
$$
It was shown in~\cite[Theorem 8]{VNB} that $2\leq \rho_{A_n}\leq 4$ for all $n$, so we have $2\leq \rho \leq 4$. This was improved in~\cite[Proposition 7.98]{Juge}, where it is shown that the sequence $\rho_{A_n}$ is non-decreasing, and that its limit $\rho$ satisfies $2.5< \rho < 4$.

We will now see that we can determine $\rho$ with arbitrary precision. For this sake, we will first describe two concepts of analytical nature: the partial theta function and the KLV-constant $q_{\infty}$.

\subsection{Topics from Real Analysis} \label{Analysis}
\subsubsection{The partial theta function} \label{SS:partial}

We start by recalling the formal power series:
$$
    f(x,y)=\sum_{k=0}^{\infty}{y^{k \choose 2}x^k}.
$$

In Sokal's paper \cite{Sokal} this series is called the {\it partial theta function}, a name that comes from the relation with the classical Jacobi theta function. It is noteworthy that this name is sometimes also used for the more general series $$
    f(x,y)=\sum_{k=0}^{\infty}{y^{Ak^2+Bk}x^k},
$$

corresponding the previous case to $A=1/2$, $B=-1/2$. We will follow here Sokal's nomenclature.

Having already appeared in Ramanujan's work \cite{AnBe09}, the partial theta function is an instance of a more general family of series which also contains the deformed exponential function (see for example \cite{Liu98}). It is also a particular example of the three-variable Rogers-Ramanujan function \cite{RoRa19}, and should not be confused with the different -but related- ``false theta functions" defined by Rogers in \cite{Rog17}.  The study of the partial theta function has become a fruitful field of research, as relations have been discovered with the so-called section-hyperbolic polynomials \cite{KoSh13}, Dirichlet series via asymptotic expansions \cite{BeKi11},  Garrett-Ismail-Stanton type identities \cite{War03, GIS99} or $q$-hypergeometric series, particularly mock modular forms \cite{BFR12}; a good guide for the theory is the year-long work of Vladimir Kostov (see for example \cite{Kos13, Kos15, Kos16}).

To our knowledge, however, no relationship had been found so far between partial theta functions and algebraic invariants, in particular coming for Group/Monoid Theory. To describe such a relation, we need to introduce the concept of leading root. Following Sokal, given a formal series $f(x,y)$ with
coefficients in a commutative unital ring $R$, there is a unique formal series $x_0(y)$ in $R[\![y]\!]$ such that $f(x_0(y),y)=0$. This series is called the {\it leading root} of $f(x,y)$, and in the case of the partial theta function defined above, it has the shape:

$$
    x_0(y)=-(1+y+2y^2+4y^3+9y^4+\cdots)
$$

The sequence of coefficients of $-x_0(y)$, which is
$$
   1, 1, 2, 4, 9, 21, 52, 133,\ldots
$$

is proved to be an increasing sequence of positive integers \cite{Sokal}, and its first 7,000 terms were computed by the author.

In \autoref{P:coefficients} below we will state that this sequence is in particular, and very surprisingly, a combinatorial invariant of Artin--Tits monoids. A different combinatorial (non-algebraic) approach to these coefficients, based on stack polyominos, can be found in unpublished work of Prellberg \cite{Pre12}.

\subsubsection{The KLV-constant $q_{\infty}$}\label{SS:constant}

In the last section of Sokal's paper~\cite{Sokal} the exponential growth of the previous sequence of coefficients is also established, as a consequence of Pringhseim theorem. This number is the inverse of the first real singularity of the leading root of the partial theta function, and turns to be a constant whose value is $3.2336366652\ldots$ As this number was first effectively computed by Katkova-Lobova-Vishnyakova in \cite{KLV03} and seems not to have a standard name, we will call it here the KLV-constant and denote it by $q_{\infty}$, as these authors did.

The KLV-constant appeared in the context of a long-standing problem, very easy to formulate. Consider the series $$g_a(z)=\sum_{k=0}^{\infty} \frac{z^k}{a^{k^2}},$$ for $a>1$. The goal is to find the smallest $a$ such that $g_a(z)$ has only real roots. This problem was first undertaken by Hardy \cite{Har04}, who proved that $a^2\geq 9$ was a sufficient condition. Afterwards, different authors attacked the problem and lowered this bound, as for example P\'{o}lya-Szeg\H{o} \cite{PoSz76} or Craven-Csordas \cite{CrCs05}, who reached a bound of 3.42. The final solution finally appeared in 2003 in \cite[Theorem 4]{KLV03}, as the mentioned authors proved that $a^2\geq q_{\infty}$ is a \emph{necessary} and sufficient condition for the series $g_a(z)$ to have only real roots. Moreover, their computation provided a way of approximating $q_{\infty}$ with arbitrary precision.

%For every $n\leq 2$, denote $q_n(g_a(z))=\frac{a_{n-1}^2}{a_{n-2}a_n}$, and assume
We will show in the next section that $q_{\infty}=\rho$. This result permits in particular a completely unexpected description of the KLV-constant in terms of growths of monoids.

\subsection{Growth rates of braid monoids and the partial theta function}

Recall from~\autoref{C:M recurrence relation type A} that, for the Artin--Tits monoid of type $A_n$, the numbers $m_{k,i}$ satisfy the following recurrence relation:
$$
   m_{k,i}=m_{k,i-1}+\sum_{j=i}^{n}(-1)^{j-i}m_{k-{j-i+2 \choose 2},j+1},
$$
where $m_{0,1}=1$.

Using this recurrence relation, we compute the table containing the numbers $m_{k,i}$ for $k\geq 0$ and $i=1,\ldots,n+1$, as in \autoref{F:table A_3}.

We will now compare the tables corresponding to distinct values of $n$. To give an example, in \autoref{F:table A_5_A_6} we can see the first 9 rows of the tables corresponding to $A_5$ and $A_6$.

\begin{figure}[ht]
\begin{center}
{\tabcolsep 2pt
\small
\begin{tabular}{|@{}c@{}||c|c|c|c|c|c|}
\hline $\begin{array}{ccc} & & \\ &  & i \\ & k &  \end{array}$ & 1 & 2 & 3 & 4 & 5 & 6 \\
\hline
\hline 0 & \bf 1 & \bf 1 & \bf 1 & \bf 1 & \bf 1 & \bf 1 \\
\hline 1 & \bf 1 & \bf 2 & \bf 3 & \bf 4 &\bf  5 & 5  \\
\hline 2 & \bf 2 & \bf 5 & \bf 9 & \bf 14 & 19 & 19 \\
\hline 3 & \bf 4 & \bf 12 & \bf 25 & 43 & 62 & 62 \\
\hline 4 & \bf 9 & \bf 30 & 68 & 125 & 187 & 187 \\
\hline 5 & \bf 21 & 75 & 181 & 349 & 536 & 536 \\
\hline 6 & 51 & 190 & 478 & 952 & 1488 & 1488 \\
\hline 7 & 126 & 484 & 1254 & 2555 & 4043 & 4043 \\
\hline 8 & 317 & 1241 & 3279 & 6786 & 10829 & 10829 \\
\hline
\end{tabular}
\qquad
\begin{tabular}{|@{}c@{}||c|c|c|c|c|c|c|}
\hline $\begin{array}{ccc} & & \\ &  & i \\ & k &  \end{array}$ & 1 & 2 & 3 & 4 & 5 & 6 & 7\\
\hline
\hline 0 & \bf 1 & \bf 1 & \bf 1 & \bf 1 & \bf 1 & \bf 1 & \bf 1\\
\hline 1 & \bf 1 & \bf 2 & \bf 3 & \bf 4 &\bf  5 & \bf 6 & 6  \\
\hline 2 & \bf 2 & \bf 5 & \bf 9 & \bf 14 & \bf 20 & 26 & 26 \\
\hline 3 & \bf 4 & \bf 12 & \bf 25 & \bf 44 & 69 & 95 & 95 \\
\hline 4 & \bf 9 & \bf 30 & \bf 69 & 132 & 221 & 316 & 316 \\
\hline 5 & \bf 21 & \bf 76 & 188 & 383 & 673 & 989 & 989 \\
\hline 6 & \bf 52 & 197 & 512 & 1091 & 1985 & 2974 & 2974 \\
\hline 7 & 132 & 517 & 1393 & 3068 & 5726 & 8700 & 8700 \\
\hline 8 & 343 & 1373 & 3794 & 8557 & 16268 & 24968 & 24968 \\
\hline
\end{tabular}}
\end{center}
\caption{Tables containing $m_{k,i}$ for the Artin--Tits monoids $A_5$ and $A_6$, for $k\leq 8$.}
\label{F:table A_5_A_6}
\end{figure}

In \autoref{F:table A_5_A_6} we have boldfaced the elements $m_{k,i}$ such that $k+i\leq n+1$. We will call them {\it stabilized entries}. From the recurrence relation in \autoref{C:M recurrence relation type A}, it is easy to see that each stabilized entry in the table for $A_n$ is computed using exactly the same values as the corresponding element in the table for $A_{n+1}$. Hence, if $m_{k,i}$ is a stabilized entry for some $A_n$, the value $m_{k,i}$ will be the same in the table corresponding to $A_{n_1}$, for every $n_1>n$.
In other words, the table stabilizes when $n$ tends to infinity: The value of each $m_{k,i}$ will become constant. By abuse of notation, we will denote $m_{k,i}$ the value of $m_{k,i}$ for $n\geq k+i-1$, that is, when its value has stabilized.

The value of the stabilized $m_{k,i}$ can be computed as follows:
\begin{equation}\label{E:recursion}
   m_{k,i}=m_{k,i-1}+\sum_{j=i}^{\infty}(-1)^{j-i}m_{k-{j-i+2 \choose 2},j+1}
\end{equation}
Notice that this sum is always finite, as the value $k-{j-i+2 \choose 2}$ must be non-negative in order to produce a nonzero summand.

There is a nice way to understand the {\it limit table}, which contains the numbers which are already stabilized.

Notice that if $n_1<n_2$, there is a natural inclusion $A_{n_1}\subset A_{n_2}$. The direct limit of the monoids $\{A_n\}_{n\geq 1}$ with respect to these natural inclusions is known as $A_{\infty}$, the braid monoid on an infinite number of strands. It has the same presentation as the usual braid monoid $A_n$, but with an infinite number of generators:
$$
A_{\infty}=\left\langle a_1,a_2,\ldots \; \left| \begin{array}{cl} a_ia_j=a_ja_i, & |i-j|>1 \\ a_ia_ja_i=a_ja_ia_j, & |i-j|=1 \end{array}   \right.\right\rangle
$$

If $M=A_{\infty}$ we can define, for $k\geq 0$ and $i\geq 1$, the number $m_{k,i}'=\left|M_k^{(i)}\right|$ as the number of elements in $M$ having length $k$, whose lex-representative starts with a letter from $\{a_1,\ldots,a_i\}$. Notice that $M_k$ is an infinite set, so we need to show that $m_{k,i}'$ is well defined.

\begin{proposition}~\label{P:stabilized}
For $k\geq 0$ and $i\geq 1$, the number $m_{k,i}'$ is well defined, and it coincides with the stabilized entry $m_{k,i}$ in the table for $A_n$, for every $n\geq k+i-1$.
\end{proposition}

\begin{proof}
Let $M=A_{\infty}$ and let $b\in M_k^{(i)}$. The lex-representative of $b$ starts with a letter from $\{a_1,\ldots,a_i\}$.  We will show that a word representing $b$ cannot contain the letter $a_t$ for $t\geq k+i$.

Recall that all words representing $b$ have length $k$, as the relations in $A_{\infty}$ are homogeneous. Also, all words representing $b$ involve the same set of letters, as this set cannot be modified by applying a relation (no new letter can appear, and no letter can dissapear).

Suppose that some word $w$ representing $b$ contains the letter $a_t$, for some $t\geq k+i$. We know that some letter from $\{a_1,\ldots,a_i\}$ appears in $w$, as it appears in the lex-representative of $b$. But $w$ has length $k$, and one of its letters already belongs to $\{a_1,\ldots,a_i\}$, so $w$ cannot involve all letters from the set $\{a_{i+1},\ldots,a_{i+k}\}$. Hence, there is some $a_j\in \{a_{i+1},\ldots,a_{i+k}\}$ which does not appear in $w$.

Let $I=\{a_1,\ldots, a_{j-1}\}$ and $J=\{a_{j+1},a_{j+1},\ldots\}$. We know that the letters in $w$ belong to $I\cup J$, and that $w$ has letters from both sets.  But every element in $I$ commutes with every element in $J$. This means that we can obtain, from $w$, a word representing $b$ having the form $w_Jw_I$, where $w_J$ only involves letters from $J$, and $w_I$ only involves letters form $I$. By hypothesis, $w_J$ is a nonempty word, so $b$ admits some element from $J$ as a prefix, and this contradicts that $b\in M_k^{(i)}$.

Therefore, every word representing $b\in M_k^{(i)}$ involves only letters from $\{a_1,\ldots,a_{k+i-1}\}$. Hence, the lex-representatives of $b$ in $A_{\infty}$ and $A_n$ coincide, for every $n\geq k+i-1$.
\end{proof}

We remark that, although the growth rate in $A_{\infty}$ does not make sense, as it is not a finitely generated monoid (the number of elements of length one is already infinite), the numbers $m_{k,i}'$ in the {\it limit table} are all well defined. As $m_{k,i}'=m_{k,i}$, we will denote these numbers by $m_{k,i}$, from now on.

Now we will find a new way to describe each stabilized number $m_{k,i}$, depending only on the elements $m_{t,1}$ for $t=0,\ldots,k-1$. For that purpose, we need some results from braid theory.

\begin{proposition}
Let $b\in A_n$ be a positive braid. For every $j\leq 1$, there is a unique maximal braid $\gamma_j$ such that $\gamma_j\preccurlyeq b$ and $\gamma_j$ involves only the generators $a_j,\ldots,a_n$.
\end{proposition}

\begin{proof}
Let $\Delta_{[j,n]}=a_j\vee \cdots \vee a_n\in A_n$. It is well-known that a positive braid $\gamma\in A_n$ can be expressed as a word in the generators $\{a_j,\ldots,a_n\}$ if and only if it is a prefix of $(\Delta_{[j,n]})^m$ for some $m>0$.  Actually, if $\gamma$ has length $t$, then $\gamma$ can be expressed as a word in the generators $\{a_j,\ldots,a_n\}$ if and only if it is a prefix of $(\Delta_{[j,n]})^t$.

Let $k$ be the length of the positive braid $b$. By the above arguments, a prefix $\gamma\preccurlyeq b$ can be written as a word in $\{a_j,\ldots,a_n\}$ if and only if $\gamma\preccurlyeq (\Delta_{[j,n]})^k$. Therefore, the set of prefixes of $b$ which involve only the generators $a_j,\ldots,a_n$ is the set of common prefixes of $b$ and $(\Delta_{[j,n]})^k$.  Since $A_n$ is a lattice with respect to the prefix order, it follows that this set has a maximal element (with respect to $\preccurlyeq$), namely $\gamma_j=b\wedge (\Delta_{[j,n]})^k$.

Notice that $\gamma_j$ is also maximal in terms of length: it is the biggest prefix of $b$ which involves only the generators $a_j,\ldots,a_n$.
\end{proof}

\begin{proposition}\label{P:decomposition}
Let $b\in A_{n}$ be a positive braid. There is a unique decomposition $b=b_n b_{n-1}\cdots b_1$, so that $b_n b_{n-1}\cdots b_j$ is the biggest prefix of $a$ which involves the generators $a_j,\ldots a_n$.
\end{proposition}

\begin{proof}
This is clear form the previous result. We just define $\gamma_{n+1}=1$ and $\gamma_j = b \wedge (\Delta_{[j,n]})^k$ for $j=1,\ldots,n$, where $k$ is the length of $b$. Then
$$
  1=\gamma_{n+1}\preccurlyeq \gamma_n \preccurlyeq \cdots \preccurlyeq \gamma_1 = b.
$$
We must define each $b_j$ so that $b_n\cdots b_j=\gamma_j$. Hence $b_j$ is the only element such that $\gamma_{j+1}b_j=\gamma_j$.
\end{proof}

\begin{proposition}\label{P:lex_of_factors}
Let $b\in A_n$ be a positive braid, and let $b=b_n b_{n-1}\cdots b_1$ be the unique decomposition described in \autoref{P:decomposition}. For every $j=1,\ldots,n$, the element $b_j$ involves generators from $\{a_j,\ldots,a_n\}$ and, if $b_j\neq 1$, its lex-representative starts with $a_j$.
\end{proposition}

\begin{proof}
Let $k$ be the length of $b$. By construction, $b_j$ is the only element such that $\gamma_{j+1}b_j=\gamma_j$, where $\gamma_j=b\wedge \Delta_{[j,n]}^k$ for every $j=1,\ldots,n$. Since $b_j$ is a suffix of $\gamma_j$, it only involves generators from $\{a_j,\ldots,a_n\}$.

Suppose that $b_j\neq 1$. From the above paragraph, the lex-representative of $b_j$ can only start with a generator from $\{a_j,\ldots,a_n\}$. We must then show that $b_j$ only admits the generator $a_j$ as prefix. Suppose this is not the case. Then $a_i\preccurlyeq b_j$ for some $i>j$. But then $\gamma_{j+1}a_i$ is a prefix of $b$ which only involves generators from $\{a_{j+1},\ldots,a_n\}$, contradicting the maximality of $\gamma_{j+1}$.
\end{proof}

We can finally give a new interpretation of the numbers $m_{k,i}$ corresponding to the monoid $M=A_{\infty}$. Recall that $M_k^{(t)}$ is the set of elements of length $k$ in $A_{\infty}$ whose lex-representative starts with $a_j$, for some $j\in \{1,\ldots,t\}$. On the other hand, let $P_k^{(t)}$ be the set of $t$-uples $(c_t,\ldots ,c_1)\in \mathcal (A_{\infty})^t$ such that $|c_1|+\cdots+|c_t|=k$ and, for $i=1,\ldots,t$, either $c_i=1$ or the lex-representative of $c_i$ starts with $a_1$.

\begin{proposition}\label{P:bijection}
Given $k\geq 0$ and $t\geq 1$, the sets $M_k^{(t)}$ and $P_k^{(t)}$ have the same size.
\end{proposition}

\begin{proof}
If $b\in M_k^{(t)}$, then $b\in A_{N}$ for some big $N$ (actually, we can take $N=t+k-1$).

By \autoref{P:decomposition}, there is a unique decomposition $b=b_N b_{N-1}\cdots b_1$, so that $b_N\cdots b_j$ is the biggest prefix of $b$ which involves the generators $a_j,\ldots,a_N$. Now, since the lex representative of $b$ starts with $a_i$ for some $i\leq t$, it follows that there cannot be a nontrivial prefix of $b$ involving the generators $a_{t+1},\ldots,a_N$. Therefore $b_j=1$ for $j>t$, and we just have $b=b_t b_{t-1}\cdots b_1$.

Let $f: A_{\infty}\rightarrow A_{\infty}$ be the shifting homomorphism which sends $a_i$ to $a_{i+1}$ for every $i\geq 1$. Notice that $f$ preserves the length of every element and, if the lex-representative of an element $b$ starts with $a_i$, the lex-representative of $f^t(b)$ starts with $a_{t+i}$.

It is clear that $f^{-j}(c)$ is defined if $c\in A_{\infty}$ involves only generators $a_i$ for $i>j$.
 By \autoref{P:lex_of_factors}, $f^{1-j}(b_{j})$ is well defined and, if it is nontrivial, its lex-representative starts with $a_1$.

We can then define the following map:
$$
   \begin{array}{rccl} \varphi: & M_k^{(t)} & \longrightarrow & P_k^{(t)} \\
                                  & b & \longmapsto & (c_t,\ldots,c_1),
   \end{array}
$$
where $c_j=f^{1-j}(b_j)$ for $j=1,\ldots,t$. We see that $\varphi$ is well defined from the above arguments, and also because $|c_t|+\cdots+|c_1| = |b_t|+\cdots+|b_1| = |b|=k$.

Now let us define the following map:
$$
   \begin{array}{rccl} \psi: & P_k^{(t)} & \longrightarrow & M_k^{(t)} \\
                              &    (c_t,\ldots,c_1) & \longmapsto & b=b_t\cdots b_1,
   \end{array}
$$
where $b_j=f^{j-1}(c_j)$ for $j=1,\ldots,t$. See \autoref{F:bijection} for an example. We will show that $\psi$ is well defined, and it is the inverse of $\varphi$, so both maps are bijections.

\begin{figure}[ht]
\begin{center}
\includegraphics{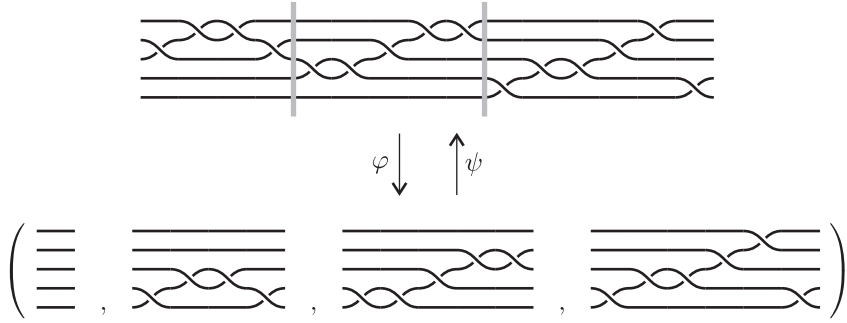}
\end{center}
\caption{A braid $b=b_4b_3b_2b_1=(1)(a_3a_4a_4a_3)(a_2a_2a_3a_4a_4)(a_1a_2a_2a_3a_4a_1)\in M_k^{(4)}$, and its corresponding 4-tuple $\varphi(b)=(1,\; a_1a_2a_2a_1,\; a_1a_1a_2a_3a_3,\; a_1a_2a_2a_3a_4a_1)\in P_k^{(4)}$.}
\label{F:bijection}
\end{figure}

It is clear that $b$ is a well-defined positive braid, of length $|b_t|+\cdots+|b_1|=|c_t|+\cdots+|c_1|=k$. Also, for every $j=0,\ldots,t-1$, the element $b_t\cdots b_{j+1}$ involves only generators $a_i$ for $i>j$.

Let us show that for $j=1,\ldots,t$, the braid $b_j\cdots b_1$ cannot start with $a_i$ for $i>j$. If $j=1$ this is clear, as $b_1=c_1$ either is trivial or its lex-representative starts with $a_1$. Suppose that $j>1$ and the claim is true for smaller values of $j$. If $a_i\preccurlyeq b_j\cdots b_1$ for some $i>j$, then $a_i\vee b_j \preccurlyeq b_j\cdots b_1$. By construction, $b_j$ cannot start with $a_i$, hence $a_i\vee b_j=b_j d$ for some nontrivial positive braid $d$. Since $a_i$ and $b_j$ involve generators of index at least $j$, the braid $d$ also involves generators of index at least $j$. But we have $b_j d \preccurlyeq b_j\cdots b_1$, hence $d\preccurlyeq b_{j-1}\cdots b_1$, which implies that $b_{j-1}\cdots b_1$ can start with a generator of index at least $j$, a fact that contradicts the induction hypothesis. The claim is then shown.

It follows that $b_t\cdots b_{j+1}$ is the biggest prefix of $b$ involving generators $a_i$ for $i>j$. This has two consequences. Firstly, the lex-representative of $b_t\cdots b_1$ starts with $a_i$ for $i\leq t$, hence $b\in M_k^{(t)}$, and $\psi$ is well defined. Secondly, $b_t\cdots b_1$ is precisely the unique decomposition of $b$ described in \autoref{P:decomposition}, which is used to define $\varphi$. So $\varphi(\psi((c_t,\ldots,c_1)))=\varphi(b)=(f^{1-t}(b_t),f^{2-t}(b_{t-1}),\ldots, f^0(b_1)) = (c_t,\ldots,c_1)$. As $\psi\circ \varphi$ is clearly equal to the identity map, it follows that $\psi$ is the inverse of $\varphi$, as we wanted to show.
\end{proof}

From \autoref{P:bijection}, we can describe the number $m_{k,t}$ directly from the numbers $m_{l,1}$ for $l\leq k$. This will give us the desired connection of these numbers with the partial theta function $f(x,y)$.

Recall that we denote $m_{k,t}=\left| M_k^{(t)}\right|$.  Now let
$$
\xi_0(y)=\sum_{k=0}^{\infty}{m_{k,1}y^k}= 1 + y + 2y^2 + 4y^3 + 9y^4 + 21 y^5+\cdots
$$
The coefficients of $\xi_0(y)$ are the numbers in the first column of the {\it limit table} containing the numbers $m_{k,i}$. Let us show that the series determined by the other columns of the table are, precisely, the powers of the series determined by the first column.

\begin{proposition}\label{P:coefficients}
For every $t>0$, one has $\displaystyle (\xi_0(y))^t= \sum_{k=0}^{\infty}{m_{k,t}y^k}$.
\end{proposition}

\begin{proof}
The $k$-th coefficient of $(\xi_0(y))^t$ is equal to:
$$
  \sum_{(k_t,\ldots,k_1)\atop k_t+\cdots+k_1=k}{m_{k_t,1}\cdots m_{k_1,1}},
$$
which is precisely the number of elements in $P_k^{(t)}$. By \autoref{P:bijection}, this is also the number of elements in $M_k^{(t)}$, so the result follows.
\end{proof}

We now denote $x_0(y)=-\xi_0(y)$, and we have the following:

\begin{proposition}
Let $\displaystyle f(x,y)=\sum_{n=0}^{\infty}{y^{n\choose 2}x^n}$. Then $f(x_0(y),y)=0$.
\end{proposition}

\begin{proof}
The expression $f(x_0(y),y)$ is a power series in the variable $y$. We need to show that all coefficients are zero.

The coefficient of $y^0$ in $f(x_0(y),y)$ comes from the values $n=0$ and $n=1$, and it is $1-1=0$, as desired.

Let $k> 0$. By \autoref{P:coefficients}, the coefficient of $y^{k}$  in $f(x_0(y),y)$ is
$$
    -m_{k,1}+m_{k-1,2}-m_{k-{3 \choose 2},3}+m_{k-{4\choose 2},4}-\cdots
$$
where there are summands as long as the first subindex of $m_{k-{r\choose 2},r}$ is non-negative. This sum equals zero by~(\ref{E:recursion}) for $i=1$, and the result follows.
\end{proof}

The following result follows immediately:

\begin{theorem}\label{T:Sokal_sequence}
Let $x_0(y)$ be the only solution to the classical partial theta function $\displaystyle \sum_{k=0}^{\infty}{y^{k \choose 2}x^k}$, and let $\xi_0(y)=-x_0(y)=1+y+2y^2+4y^3+9y^4+\cdots$ For every $k\geq 0$, the coefficient of $y^k$ in the series $\xi_0(y)$ is equal to the number of braids of length $k$, in the monoid $A_{\infty}$, whose maximal lexicographic representative starts with $a_1$.
\end{theorem}

\subsection{Limit of growth rates of braid monoids}

We are considering the partial theta function $f(x,y)=\displaystyle \sum_{k=0}^{\infty}{y^{k \choose 2}x^k}$, whose only nontrivial root is the series $x_0(y)$. The coefficients of the series $\xi_0(y)=-x_0(y)=1+y+2y^2+4y^3+9y^4+\cdots$ form the sequence
$$
   (L_k)_{k\geq 0}=(1,1,2,4,9,21,52,\ldots).
$$
The growth rate of this sequence is known to be the KLV-constant $q_{\infty}$~\cite{Sokal}, which can be computed with arbitrary precision:

\begin{theorem}\label{T:Sokal}{\rm \cite{Sokal}}
Let $(L_k)_{k\geq 0}$ be the sequence of coefficients of $\xi_0(y)$. Then its growth rate is:
$$
   \lim_{k\to\infty}{\sqrt[k]{L_k}}=\lim_{k\to\infty}{\frac{L_{k+1}}{L_k}}=q_\infty= 3.233636\ldots
$$
\end{theorem}

Notice that the above result states that $(L_k)_{k\geq 0}$ grows like $(q_\infty)^k$. In other words:
\begin{equation}\label{E:proportion}
    0<\lim_{k\to\infty}{\frac{L_k}{(q_{\infty})^k}}<\infty
\end{equation}

We want to relate the constant $q_{\infty}$ to the growth rate of the monoids $A_n$, for $n\geq 1$. Recall that we are counting the elements in $A_n$ by considering their lex-representatives (their maximal lexicographic representatives with $a_1<a_2<\cdots <a_n$). The following is an important property of this set of words:

\begin{theorem}\label{T:regular_language}{\rm \cite{GGM}}
For every $n\geq 1$, the set of lex-representatives of the braid monoid $A_n$ is a regular language.
\end{theorem}

In~\cite{GGM}, an automaton accepting this regular language is defined, having the minimal possible number of states. Moreover, in~\cite{FlGo18}, the automaton is described in detail for every $n\geq 1$, and the following result is shown:

\begin{theorem}\label{T:1/32}{\rm \cite[Corollary 5.5]{FlGo18}}
For every $n\geq 1$, the proportion of lex-representatives of length $k$, in the braid monoid $A_n$, finishing at the same state as $a_1$, tends to a limit $p>\frac{1}{32}$ when $k$ tends to infinity.
\end{theorem}

The above result can be described as follows. Let $M_n$ be the incidence matrix of the automaton accepting the regular language of lex-representatives of $A_n$. Each row (resp. column) of $M_n$ corresponds to a state of the automaton. We can assume that the first row (and also the first column) corresponds to the state determined by the word $a_1$.

By Perron-Frobenius theory, there is a unique left eigenvector $\mathbf{v_n}$ of $M_n$, all of whose coordinates are non-negative and such that the sum of these coordinates is equal to 1. The $i$th coordinate of $\mathbf{v_n}$ is precisely the limit, as $k$ tends to infinity, of the proportion of lex-representatives of length $k$ finishing at the $i$th state (see~\cite{FlGo18}). \autoref{T:1/32} states that the first coordinate of $\mathbf{v_n}$ is greater than $\frac{1}{32}$, for all $n\geq 1$.

Now let us fix some $n\geq 1$. We will denote by $m_{k,i}(A_n)$ the numbers appearing in \autoref{C:M recurrence relation type A}, which are also the numbers appearing in \autoref{F:table A_5_A_6} for $n=5$ and $n=6$. We know that the growth rate of the column $n+1$ is precisely $\rho_n$, the growth rate of the monoid $A_n$. Let us show that all the columns of the table corresponding to $A_n$ have the same growth rate.

\begin{proposition}
Let $n\geq 1$. For every $i=1,\ldots,n+1$, we have:
$$
    \lim_{k\to\infty}{\frac{m_{k+1,i}(A_n)}{m_{k,i}(A_n)}}= \rho_n.
$$
\end{proposition}

\begin{proof}
The result is trivial if $i=n+1$ or if $i=n$, as these columns contain precisely the number of braids of given length.

By definition, the number $m_{k,i}(A_n)$ is the number of lex-representatives in $A_n$ starting with a generator from $\{a_1,\ldots,a_i\}$. If we consider the incidence matrix $M_n$, and assume that the rows $1,2,\ldots,n$ correspond to the states of the lex-representatives $a_1,a_2,\ldots,a_n$ respectively, then $m_{k,i}(A_n)$ is the sum of the entries of the rows $1,2,\ldots,i$ in the matrix $(M_n)^{k-1}$.

By Perron-Frobenius theory, all rows of $(M_n)^k$ have the same growth rate as $k$ tends to infinity, which is precisely the Perron-Frobenius eigenvalue. In this case, this eigenvalue is precisely $\rho_n$ (the growth rate of the sum of the first $n$ rows). Therefore, all sequences $(m_{k,i}(A_n))_{k\geq 0}$ grow like $(\rho_n)^k$, as we wanted to show.
\end{proof}

Now we will relate the numbers $\rho_n$ to the constant $q_{\infty}$. Recall from~\cite{Juge} that ${\rho_n}$ is an increasing sequence of real numbers, whose limit we denote $\rho$:
$$
  \lim_{n\to\infty}{\rho_n}=\rho.
$$
We need the following:

\begin{definition}\label{D:c(n,mu)}
Given a real number $\mu\geq \rho$, let $\displaystyle c(n,\mu)=\sum_{k=0}^{\infty}{\left(\frac{\rho_n}{\mu}\right)^k}= \frac{1}{1-{\frac{\rho_n}{\mu}}}= \frac{\mu}{\mu-\rho_n}$.
\end{definition}

\begin{lemma}\label{L:limite}
Given $\mu\geq \rho$, one has $\displaystyle \lim_{n\to\infty}{c(n,\mu)}=\infty$ if and only if $\mu=\rho$.
\end{lemma}

\begin{proof}
It follows trivially from \autoref{D:c(n,mu)}.
\end{proof}

Now let us denote by $\mathbf e$ the (column) vector (of any desired length) which consists only of 1's. Hence, from the definition of the Perron-Frobenius eigenvector $\mathbf{v_n}$, we have $\mathbf{v_n}\cdot \mathbf e = 1$.

\begin{lemma}\label{L:c(n,mu)}
For every $\mu\geq \rho$ and every $n\geq 2$, we have:
$$
   c(n,\mu) = \mathbf{v_n} (I-\mu^{-1}M_n)^{-1} \mathbf e
$$
\end{lemma}

\begin{proof}
We know~\cite[Lemma 5.2]{FlGo18} that
$$
   (I-\mu^{-1}M_n)^{-1} = I + (\mu^{-1}M_n) + (\mu^{-1}M_n)^2 + (\mu^{-1}M_n)^3 +\cdots
$$
If we multiply any summand by $\mathbf{v_n}$ from the left, and by $\mathbf e$ from the right, we obtain
$$
   \mathbf{v_n} (\mu^{-1}M_n)^k \mathbf e = \frac{1}{\mu^k}\left(\mathbf{v_n}M_n^k\right) \mathbf e = \frac{1}{\mu^k} \left(\rho_n^k \mathbf{v_n}\right)\mathbf e = \left(\frac{\rho_n}{\mu}\right)^k.
$$
Therefore
$$
  \mathbf{v_n}(I-\mu^{-1}M_n)^{-1} \mathbf e= 1+\left(\frac{\rho_n}{\mu}\right) + \left(\frac{\rho_n}{\mu}\right)^2 +\left(\frac{\rho_n}{\mu}\right)^3 +\cdots = c(n,\mu).
$$
\end{proof}

Now recall that the number $m_{k,1}(A_n)$ is equal to the sum of the entries of the first row of $(M_n)^{k-1}$. That is, $m_{k,1}(A_n)$ is the first entry of the column vector $(M_n)^{k-1}\mathbf e$. Recall also that the sequence $\{m_{n,1}(A_n)\}_{n\geq 0}$ is equal to $\{m_{n,1}(A_{\infty})\}_{n\geq 0}=\{L_n\}_{n\geq 0}= \{1,1,2,4,9,\ldots\}$. By \autoref{T:Sokal}, the growth rate of $\{m_{n,1}(A_n)\}_{n\geq 0}$ is equal to the constant $q_{\infty}$.

Notice that, for every $n\geq 1$ and every $k\geq 0$, we have $m_{k,1}(A_n)\leq m_{k,1}(A_{\infty})$. Hence, the sequence $\{m_{k,1}(A_n)\}_{k\geq 0}$ is dominated by the sequence $\{m_{k,1}(A_{\infty})\}_{k\geq 0}$. It follows that the growth rate of the former sequence cannot be bigger than the growth rate of the latter. In other words: $\rho_n\leq q_{\infty}$ for every $n\geq 1$. Therefore,
$$
    \rho=\lim_{n\to\infty}{\rho_n}\leq q_{\infty},
$$
and we can consider the numbers $c(n,q_\infty)$.

\begin{proposition}\label{P:limit_of_c(n,q)}
We have:
$$
   \lim_{n\to\infty}{c(n,q_{\infty})}=\infty.
$$
\end{proposition}

\begin{proof}
We know from \autoref{L:c(n,mu)} that
$$
c(n,q_{\infty})= \mathbf{v_n}I\mathbf e + \mathbf{v_n}(q_{\infty}^{-1}M_n)\mathbf e + \mathbf{v_n}(q_{\infty}^{-1}M_n)^2\mathbf e + \mathbf{v_n}(q_{\infty}^{-1}M_n)^3\mathbf e +\cdots
$$
That is,
$$
c(n,q_{\infty})= 1+ \frac{\mathbf{v_n}M_n\mathbf e}{q_{\infty}} + \frac{\mathbf{v_n} M_n^2\mathbf e}{q_{\infty}^2} + \frac{\mathbf{v_n} M_n^3\mathbf e}{q_{\infty}^3}+\cdots
$$
We now recall from \autoref{T:1/32} that the first coordinate of $\mathbf{v_n}$ is greater than $\frac{1}{32}$. Therefore, if we denote $\mathbf e_1$ the first row of the identity matrix, we have, for every $k>0$:
$$
   \mathbf{v_n} M_n^k\mathbf e > \frac{1}{32} \left(\mathbf e_1 M_n^k \mathbf e \right) = \frac{m_{k+1,1}(A_n)}{32}.
$$
The above inequality holds since all the coordinates of vectors and matrices involved are nonnegative. Recall also that $m_{1,1}(A_n)=1$, so $1>\frac{1}{32}=\frac{m_{1,1}(A_n)}{32}$.

Finally, we obtain:
$$
c(n,q_{\infty})> \frac{m_{1,1}(A_n)}{32}  + \frac{m_{2,1}(A_n)}{32\: q_{\infty}} + \frac{m_{3,1}(A_n)}{32 \:q_{\infty}^2} + \frac{m_{4,1}(A_n)}{32 \:q_{\infty}^3} +\cdots
$$

Now recall that for every $k\leq n$, we have $m_{k,1}(A_n)=m_{k,1}(A_k)$. Also, $q_\infty>3$, so we can divide the above expression by $q_{\infty}$ and truncate at the $n$th term, to obtain:
$$
c(n,q_{\infty})> \frac{m_{1,1}(A_1)}{32\: q_{\infty}}  + \frac{m_{2,1}(A_2)}{32\: q_{\infty}^2} + \frac{m_{3,1}(A_3)}{32 \:q_{\infty}^3} +\cdots + \frac{m_{n,1}(A_n)}{32 \:q_{\infty}^n}.
$$
The right hand side is the truncation at $n$ of the infinite sum:
$$
  \frac{1}{32}\sum_{k=1}^{\infty}{\frac{m_{k,1}(A_k)}{q_{\infty}^k}}
$$
Since the numerators correspond to the sequence $\{L_k\}_{k\geq 0}$ whose growth rate is $q_{\infty}$, it follows that the fractions tends to a positive real number, so the above infinite sum does not converge. In other words, its truncations tend to infinity, and this implies that $c(n,q_{\infty})$ also tends to infinity, as $n$ grows.
\end{proof}

We can finally show the main result of this section.

\begin{theorem}\label{T:growth_limit}
Let $\displaystyle \rho=\lim_{n\to \infty}{\rho_{A_n}}$. Then $\rho=3.23363\ldots $ is the growth rate of the coefficients of $\xi_0(y)$. That is, $\rho$ is equal to the KLV-constant $q_{\infty}$.
\end{theorem}

\begin{proof}
By \autoref{P:limit_of_c(n,q)}, we have that $\lim_{n\to\infty}{c(n,q_{\infty})}=\infty$. By \autoref{L:limite}, this can only happen if $q_{\infty}=\rho$.
\end{proof}

We finish this paper with a question concerning the remaining Artin--Tits monoids:

{\bf Question:} {\it Is it true that $\rho(A_n)= \rho(B_n)= \rho(D_n)= q_{\infty} $?}

Using injective maps $A_n\rightarrow B_{n+1}$ and $A_n\rightarrow D_{n+1}$, which send $a_i$ to $a_i$, it is easy to prove that $\rho(B_n)\leq q_{\infty}$; however, our methods are unable to state the opposite inequalities.

%----------------------------------------------------------------------------

%|<------------------------------------------------------------------------>|

{\bf Ram\'on Flores.}\\
ramonjflores@us.es\\
Depto. de Geometr\'{\i}a y Topolog\'{\i}a. Instituto de Matem\'aticas (IMUS). \\
Universidad de Sevilla.  Av. Reina Mercedes s/n, 41012 Sevilla (Spain).

\medskip

{\bf Juan Gonz\'alez-Meneses.}\\
meneses@us.es\\
Depto. de \'Algebra. Instituto de Matem\'aticas (IMUS). \\
Universidad de Sevilla.  Av. Reina Mercedes s/n, 41012 Sevilla (Spain).

\end{document}